\documentclass{aic}

\usepackage[utf8]{inputenc}

\usepackage{mathtools}

\usepackage{cleveref}

\usepackage{tikz}
\usepackage{tikz-cd}
\usetikzlibrary{calc}

\theoremstyle{plain}
\newtheorem{theo}{Theorem}
\crefname{theo}{Theorem}{Theorems}
\newtheorem{lem}[theo]{Lemma}
\crefname{lem}{Lemma}{Lemmas}
\newtheorem{res}[theo]{Result}
\crefname{res}{Result}{Results}
\newtheorem{cor}[theo]{Corollary}
\crefname{cor}{Corollary}{Corollarys}

\crefname{prop}{Proposition}{Propositions}
\newtheorem{fact}[theo]{Fact}

\newtheorem{conj}[theo]{Conjecture}

\theoremstyle{definition}

\crefname{exa}{Example}{Examples}

\crefname{prob}{Problem}{Problems}

\theoremstyle{remark}

\crefname{section}{Section}{Sections}
\crefname{figure}{Figure}{Figures}

\newtheorem{claim}{Claim}[theo]
\crefname{claim}{Claim}{Claims}
\renewcommand{\theclaim}{\arabic{claim}}
\newenvironment{claimproof}[1][\proofname\ of Claim \theclaim]{%
  \proof[#1]%
  
}{\endproof}

\newcommand{\icase}[1]{\par\medskip\noindent\textit{Case #1: }}

\renewcommand{\phi}{\varphi}
\renewcommand{\epsilon}{\varepsilon}

\newcommand{\NN}{{\mathbb N}}

\newcommand{\ZZ}{{\mathbb Z}}

\newcommand{\CB}{{\mathcal B}}

\newcommand{\CO}{{\mathcal O}}
\newcommand{\CP}{{\mathcal P}}
\newcommand{\CQ}{{\mathcal Q}}

\newcommand{\CS}{{\mathcal S}}

\newcommand{\CZ}{{\mathcal Z}}

\newcommand{\Gammad}{\Gamma_d}
\newcommand{\Thetad}{\Theta_d}

\newcommand{\limgraphG}{\overline{G}}

\newcommand{\id}{\operatorname{id}}

\renewcommand{\hat}{\widehat}
\renewcommand{\tilde}{\widetilde}
\renewcommand{\hat}{\widehat}

\newcommand{\aut}{\operatorname{Aut}}

\newcommand{\dist}{\operatorname{dist}}

\newcommand{\LeftMostSep}{\mathop{\mathrm{sep}}}
\newcommand{\LevelSet}{\mathop{\mathrm{Lev}}}

\DeclareMathOperator{\Aut}{Aut}

\DeclareMathOperator{\Diam}{diam}
\DeclareMathOperator{\cov}{cov}
\DeclareMathOperator{\Had}{Had}
 
\aicAUTHORdetails{
  title = {Automorphism Groups of Graphs of Bounded Hadwiger Number}, 
  author = {Martin Grohe, Pascal Schweitzer, and Daniel Wiebking},
  plaintextauthor = {Martin Grohe, Pascal Schweitzer, Daniel Wiebking},
  keywords = {automorphism groups of finite graphs, Hadwiger numbers, edge-transitive graphs, pointwise graph limits},
}

\aicEDITORdetails{
   year={2025},
   number={9},
   received={20 February 2023},   
   published={10 October 2025},  
   doi={10.19086/aic.2025.9},     
}   

\begin{document}

\begin{frontmatter}[classification=text]
\title{Automorphism Groups of Graphs of Bounded Hadwiger Number\titlefootnote{The research leading to these results has received funding from the European Research Council (ERC) under the European Union's Horizon 2020 research and innovation programme (EngageS: grant agreement No.~820148) and by the German Research Foundation DFG Koselleck Grant GR 1492/14-1.}} 
\author[mgro]{Martin Grohe}
\author[pschw]{Pascal Schweitzer}
\author[dwieb]{Daniel Wiebking}

\begin{abstract}
We determine the structure of automorphism groups of finite graphs of bounded Hadwiger number. Our proof includes a structural analysis of finite edge-transitive graphs.

In particular, we show that for connected, $K_{h+1}$-minor-free, edge-transitive, twin-free, finite graphs the non-abelian composition factors of the automorphism group have bounded order.

We use this to show that the automorphism groups of finite graphs of bounded Hadwiger number are obtained by repeated group extensions using abelian groups, symmetric groups and groups of bounded order.
\end{abstract}
\end{frontmatter}

\section{Introduction}
Frucht's classic theorem shows that every abstract finite group is the automorphism group of a finite graph, and in fact the graph can even be required to be connected and 3-regular~\cite{MR32987}. We say that the class of finite connected 3-regular graphs is \emph{universal}.
In the years before 1990, various classes of graphs were proven to not be universal, typically by providing a structure theorem for automorphism groups of graphs in the class.
Along these lines, we know that automorphism groups of finite trees are iterated direct and wreath products of symmetric groups~\cite{MR1577579}. More generally, 
Babai gave a classification of automorphism groups of finite planar graphs~\cite{MR0302494,MR0371715}. An early survey of known results, for example including lattices, designs, and strongly regular graphs, is given by Babai and Goodman~\cite{MR1227125}.
 
In this paper, we are interested in classes of finite graphs closed under taking minors. A graph class is minor-closed if it is closed under edge contractions and under taking subgraphs. 

Examples of minor-closed graph classes are the classes of trees, planar graphs, bounded genus graphs (i.e., graphs embeddable without crossings on a fixed surface) and graphs of bounded treewidth. All of these classes play central, recurring roles in graph theory. Their study dates back at least to Wagner's theorem~\cite{MR1513158} from 1937, which states that a graph is planar exactly if it contains neither a complete graph on five vertices ($K_5$) nor a complete bipartite graph with the two parts each having three vertices~($K_{3,3}$) as a minor.

The Hadwiger number of a graph is the order of a largest complete graph obtainable by edge contractions. 
If a minor-closed class of finite graphs is non-trivial (i.e., it is not just the class that contains every graph) then it excludes some graph and therefore it excludes some complete graph.
Thus, a minor-closed class of finite graphs is non-trivial precisely if the Hadwiger number is bounded.
The study of minor-closed graph classes therefore reduces to studying classes of bounded Hadwiger number. In other words, the Hadwiger number is bounded for all classes mentioned above such as trees, planar graphs and bounded genus graphs. We investigate the structure of the automorphism groups of graphs in such classes.

In 1974, Babai showed that a non-trivial minor-closed graph class is not universal by proving that large alternating groups cannot appear as automorphism groups of graphs of bounded Hadwiger number~\cite{MR332554}. In fact, he also showed that for a sufficiently large prime~$p$ the group~$\mathbb{Z}_p^3$ does not appear as a subgroup of a simple group represented by these graphs~\cite{MR389656}.

Babai also proved there are strong restrictions for the automorphisms of strongly regular graphs~\cite{MR3359489}. However, no structural description of groups represented by graphs of bounded Hadwiger number has been available.
Progress towards this was made independently by Babai~\cite{DBLP:journals/jgt/Babai91} and Thomassen~\cite{MR1040045}. They investigated
vertex-transitive graphs of bounded Hadwiger number and in particular showed that for~$g\geq 3$ there are only finitely many vertex-transitive graphs of genus~$g$.

Regarding graphs of bounded Hadwiger number, one of Babai's central theorems shows that there is a function~$f$ so that almost all finite vertex-transitive graphs of Hadwiger number at most~$h$ can be embedded on the torus or are~$(f(h),f(h))$-ring-like. The latter means that there is a system of blocks of imprimitivity each of size at most~$f(h)$ which has a circular ordering and edges only connect blocks of distance~$f(h)$ in this circular ordering.
This theorem is mentioned as early as 1993. While a formal proof has not appeared, the theorem has been mentioned in various publications (\cite{MR1227125,MR1373683,CameronSurvey,1110.4885}\footnote{Reference~\cite{1110.4885} is the extended arXiv version of a paper~\cite{DBLP:journals/endm/DeVosM06}. In there, the authors also claim that their results can be used to prove Babai's theorem. They refer to a paper ``A quantitative strengthening of Babai's theorem'' in preparation that does not seem to have appeared.}) often with sketches of the overall idea\footnote{We thank Laci Babai for explaining the proof idea to us and providing various pointers to the literature at the ``Symmetry vs Regularity --- 50 Years of Weisfeiler-Leman Stabilization'' conference in Pilsen in 2018.}. In any case, it seems to have been already clear at the time (see~\cite{MR1227125}) that characterizations for edge-transitive rather than vertex-transitive graphs are required in order to obtain overall structure theorems for entire graph classes.

In this paper, we determine the structure of the automorphism group of finite graphs of bounded Hadwiger number. We do so by first proving the following result for edge-transitive graphs.

\begin{res}
For connected, $K_{h+1}$-minor-free, edge-transitive, twin-free, finite graphs the non-abelian composition factors of the automorphism group have bounded order (see Theorem~\ref{thm:edge:transitive:twin:free:comp:fac:bound}).
\end{res}

This allows us to determine the automorphism group structure of bounded Hadwiger number graphs as follows.
\begin{res}

 The automorphism groups of finite graphs of bounded Hadwiger number are obtained by repeated group extensions using abelian groups, symmetric groups and groups of bounded order (see Theorem~\ref{thm:kh:minor:free:means:Theta:group}).
\end{res}

Our structure theorem resolves three of Babai's long-standing conjectures stated already in Babai's 1981 survey on the abstract group of automorphisms~\cite{MR633647}.

First, regarding composition factors, Babai conjectured the following.
\begin{conj}\label{babai:conj:1}
There is a function~$f$ such that a composition factor of the automorphism group of a finite graph of Hadwiger number~$h$ is cyclic, alternating, or has order at most~$f(h)$. 
\end{conj}

Second, regarding representability of simple groups, Babai's \emph{subcontraction conjecture} states the following:
\begin{conj}\label{babai:conj:2}
Only finitely many non-cyclic simple groups are represented by finite graphs of bounded Hadwiger number. \end{conj}
In fact, he anticipated that alternating composition factors can only appear within their corresponding symmetric factors. This is indeed the case.

Finally, third, Babai conjectured the absence of small prime factors in the automorphism group order has an impact on the possible structure of the group as follows.
\begin{conj}\label{babai:conj:3}
There is a function~$f$ with the following property.
If the order of the automorphism group of a finite graph of Hadwiger number~$h$ does not have prime factors smaller than~$f(h)$, then the automorphism group is obtained by forming repeated direct products and wreath products of abelian groups.
\end{conj}

All the conjectures follow fairly directly in the affirmative from our structural theorem.

\begin{cor}
Conjectures~\ref{babai:conj:1},~\ref{babai:conj:2}, and~\ref{babai:conj:3} are true. 
\end{cor}

Regarding our proof, we combine techniques from various
areas of graph and group theory, as well as geometry. In particular, we transfer Babai's
theorem mentioned above from vertex-transitive graphs
to edge-transitive graphs, following his approach involving sphere
packings, infinite rooted limit graphs and Archimedean tilings. We
also exploit submodularity arguments for separations and graph
covering maps.

We should highlight that previous results of this sort were neither known for graphs of bounded Hadwiger number nor for prominent special cases such as graphs of bounded treewidth and graphs of bounded genus.

\textbf{Structure of the paper.} Following the preliminaries (Section~\ref{sec:prelims})
we discuss the structure of infinite edge-transitive graphs that have two ends (Section~\ref{sec:twoEnds}). We then analyze the automorphism group structure of finite edge-transitive graphs~(Section~\ref{sec:finite:edge:trans}). We characterize these groups~(Theorem~\ref{thm:edge:transitive:twin:free:comp:fac:bound}) by treating separately the cases of 
\begin{itemize}\setlength\itemsep{-0.2em}
\item one end (Subsection~\ref{subsec:one:end}) using infinite planar graphs and Archimedean tilings,
\item infinitely many ends (Subsection~\ref{subsec:infinitely:many:ends}) using sphere packings, and
\item two ends (Subsection~\ref{subsec:two:ends}) using the structural results from Section~\ref{sec:twoEnds} and graph coverings.
\end{itemize} 

The results are combined (Subsection~\ref{subsec:comb:results}) to prove Theorem~\ref{thm:edge:transitive:twin:free:comp:fac:bound}. We then discuss how the results for edge-transitive can be used to treat the general case of more than one edge orbit~(Section~\ref{sec:multiple:edge:orbits}) yielding a general structure result for the automorphism group of finite graphs with bounded Hadwiger number (Theorem~\ref{thm:kh:minor:free:means:Theta:group}).
We finally use this theorem to prove Babai's conjectures (Section~\ref{sec:babai:conj}) and conclude (Section~\ref{sec:conclusion}).

\section{Preliminaries}\label{sec:prelims}

\paragraph{Graphs}

Unless stated otherwise, we consider undirected graphs $G=(V(G),E(G))$ consisting of a \emph{vertex set} $V(G)$ and an \emph{edge set} $E(G)\subseteq\{\{v,w\}\subseteq V(G)\mid v\neq w\}$.
An edge $\{v,w\}\in E(G)$ is also denoted $vw$.
For $n\in\NN$, we write $[n]\coloneqq\{1,\ldots,n\}$.
The \emph{distance} between $v,w\in V(G)$ in~$G$, denoted $\dist_G(v,w)$, is the length (number of edges) of a shortest path from~$v$ to~$w$.
For a vertex $v\in V(G)$, the \emph{(closed) ball centered at $v$ with radius~$t$}
is defined as $B_{t,G}(v)\coloneqq\{w\in V(G)\mid\dist_G(v,w)\leq t\}$.
The \emph{(open) neighborhood} of a vertex $v\in V(G)$ is defined as $N_G(v)\coloneqq\{w\in V(G)\mid vw\in E(G)\}$.
The \emph{(open) neighborhood} of a subset $A\subseteq V(G)$ is defined as $N_G(A)\coloneqq\bigcup_{v\in A}N_G(v)\setminus A$.
Two vertices $v,v'\in V(G)$ are called \emph{(true) twins} in $G$
if $N_G(v)=N_G(v')$, and a graph $G$ is called \emph{twin-free} if there are no distinct vertices $v\neq v'$
in $G$ that are twins.
The \emph{degree} of a vertex $v\in V(G)$ is denoted by $\deg_G(v)\coloneqq|N_G(v)|$.
A graph $G$ is called \emph{locally finite} if $\deg_G(v)$ is finite
for all vertices $v\in V(G)$.
A graph is \emph{$d$-regular} for~$d\in \mathbb{N} $, if all vertices have degree $d$ and \emph{regular} if it is~$d$-regular for some~$d$.
Similarly, a bipartite graph with bipartition $V_1,V_2$ is called \emph{$(d_1,d_2)$-biregular}
if for each~$i\in\{1,2\}$ all vertices in $V_i$ have degree $d_i$.
It is \emph{biregular} if it is $(d_1,d_2)$-biregular for some~$d_1$ and~$d_2$.
For a graph $G$,
the \emph{induced subgraph} on a subset $S\subseteq V(G)$, denoted by $G[S]$,
is the graph with vertex set $S$ and edge set $\{e\in E(G)\mid e\subseteq S\}$.
Similarly, the \emph{induced bipartite subgraph} on two disjoint subsets $V_1,V_2\subseteq V(G)$, denoted by $G[V_1,V_2]$,
is defined as the bipartite graph with bipartition $V_1,V_2$ and edge set $E(G[V_1,V_2])\coloneqq\{v_1v_2\in E(G)\mid v_1\in V_1, v_2\in V_2\}$.
We write $G-S$ to denote the induced subgraph $G[V(G)\setminus S]$.
A \emph{separator} of a connected graph is a subset of the
vertices~$S\subseteq V(G)$ for which~$G-S$ is disconnected.
On~$n$ vertices, the \emph{complete graph} is denoted $K_n$, the
\emph{cycle} $C_n$, and the path $P_n$. The complete bipartite graph on~two parts of order~$n$ is denoted~$K_{n,n}$. The \emph{Cartesian product}
of two graphs $G$ and~$H$, denoted $G\square H$,
is the graph with vertex set $V(G\square H)\coloneqq V(G)\times V(H)$
and edge set $\{(v_1,v_2)(w_1,w_2)\mid v_1w_1\in E(G)\wedge v_2=w_2$ or
$v_1=w_1\wedge v_2w_2\in E(H)\}$.
A \emph{vertex-colored graph} is a pair $G_\chi=(G,\chi)$
consisting of a graph~$G$ 
and a function $\chi\colon V(G)\to C$, called \emph{vertex coloring}, that assigns to each vertex an
element in $C$, called the \emph{color} of that vertex.

For a function $\phi$ with domain $U$ and an element $u\in U$, we usually denote the image $\phi(u)$ by $u^\phi$,
and similarly for a subset $S\subseteq U$, we write $S^\phi\coloneqq\{\phi(u)\mid u\in S\}$.
For this reason, we compose functions $\phi\colon U\to V,\psi\colon V\to W$ from left to right, i.e.,
$u^{\phi\psi}=(u^\phi)^\psi$.
Two graphs $G,H$ are \emph{isomorphic}
if there is a bijection $\phi\colon V(G)\to V(H)$
such that $vw\in E(G)$ if and only if $v^\phi w^\phi\in
E(H)$.
In this case, the bijection $\phi$ is called an \emph{isomorphism} from $G$ to $H$.
The \emph{automorphism group} of a graph, denoted $\aut(G)$, is the
group of isomorphisms from~$G$ to itself.
A graph $G$ is \emph{vertex-transitive}
if $\aut(G)$ is transitive.
A graph is \emph{edge-transitive}
if $\aut(G)$ acts transitively on $E(G)$, i.e.,
for all pairs of edges $vw,v'w'\in E(G)$ there is an automorphism $\sigma\in\aut(G)$
such that $v^\sigma w^\sigma=v'w'$ (but not necessarily $(v^\sigma,
w^\sigma)=(v',w')$).
Note that an edge-transitive graph without isolated vertices
is either vertex-transitive or bipartite with
the automorphism group acting transitively on both parts of the bipartition.
Thus, each edge-transitive graph is \emph{almost vertex-transitive}, i.e.,
the automorphism group has finitely many vertex orbits.

The notion of isomorphisms and automorphisms can naturally be extended to bipartite graphs (where the parts may not be interchanged)
and directed graphs, as well as vertex-colored graphs.
For vertex-colored graphs, an isomorphism has to preserve the vertex coloring in addition to the edge relation, i.e.,
an isomorphism $\phi$ also has to satisfy $\chi(v)=\chi(v^\phi)$ for all vertices $v$.

While our theorem is concerned with finite graphs, the proof involves infinite graphs. They arise as limits of finite graphs of bounded maximum degree. Various properties of the finite graphs, such as edge-transitivity, transfer to infinite graphs. Since we are interested in finite separators of the infinite graph, we discuss the concept of ends next.

A \emph{ray} in an infinite connected graph $G$ is a one-way infinite path $v_1,v_2,v_3,\ldots$
(in which each vertex appears at most once).
We define an equivalence relation on the set of rays, in which two rays $r,r'$ are equivalent if for every finite separator~$S$ there is a connected component of $G-S$ all but finitely many
vertices from $r$ and almost all vertices from $r'$.
The equivalence classes of this relation are called the \emph{ends} of the graph $G$.
It is known (see~\cite{hopf1943enden,DBLP:conf/soda/Babai97})
that infinite, connected, locally finite and
almost vertex-transitive graphs have either one, two or infinitely many ends.
A subset~$C$ of the vertices of~$G$ \emph{contains} a particular end, if for every ray~$r$ of the end only finitely many vertices of~$r$ are not contained in~$C$. Note that for a finite set~$S$ every end of~$G$ is contained in some connected component of~$G-S$. Conversely, if~$G$ is a connected and locally finite graph then for a finite set of vertices~$S$ every infinite component of~$G-S$ contains at least one end.
A finite set~$S$ of vertices \emph{separates two ends} if the ends are contained in different connected components of~$G-S$.

\paragraph{Graph Minors}
Let $\CB$ be a partition of $V(G)$ such that $G[B]$ is connected for all $B\in\CB$.
We define $G/\CB$ to be the graph with vertex set $V(G/\CB) \coloneqq \CB$ and
$E(G/\CB) \coloneqq \{BB' \mid \exists v \in B, v' \in B' \colon vv' \in E(G)\}$.
A graph $H$ is a \emph{minor} of $G$ if there is a partition $\CB$ into connected subsets, called \emph{branch sets}, such that $H$ is isomorphic to a subgraph $H_\CB$ of~$G/\CB$.
In this case, an isomorphism $\phi\colon V(H)\to V(H_\CB)\subseteq 2^{V(G)}$ is called a \emph{minor model}
of $H$ in $G$.
A minor is called \emph{$\aut(G)$-invariant}  (or just invariant)
if there is a minor model $\varphi\colon V(H)\rightarrow 2^{V(G)}$
so that~$V(H)^\phi$ and $E(H)^\phi$ are both~$\aut(G)$-invariant
where $E(H)^\phi\coloneqq\{v^\phi w^\phi\mid vw\in E(H)\}$.
A graph $G$ \emph{excludes $H$ as a minor} if $H$ is not a minor of $G$.
The \emph{Hadwiger number} of a graph $G$, denoted by $\Had(G)$, is the largest number $h\in\NN$ such that the complete
graph with $h$ vertices is a minor of $G$.
Equivalently, the Hadwiger number of a graph $G$ is the smallest number $h\in\NN$ such that $G$
excludes the complete graph $K_{h+1}$ as a minor.

\paragraph{Groups} We refer to~\cite{MR1409812} for basics on permutation groups. We use capital Greek letters to denote groups except for the \emph{symmetric group} on~$[d]$, which we denote by~$S_d$. The notation~$\Psi\trianglelefteq \Delta$ indicates that~$\Psi$ is a \emph{normal subgroup} of~$\Delta$. A \emph{block} of a permutation group on~$\Delta$ on a set~$V$ is a set~$B\subseteq V$ such that for all~$\delta\in \Delta$ we have~$B^\delta=B$ or~$B^\delta\cap B= \emptyset$.

If~$\Delta$ is a permutation group on a set~$V$ and~$B\subseteq V$ is invariant under~$\Delta$, then the group \emph{induced by~$\Delta$ on~$B$} is the group of permutations of~$B$ that are restrictions of elements of~$\Delta$. The \emph{wreath product} of a base group~$\Delta$ with a top group~$\Psi$ with respect to the product action is denoted~$\Delta\wr \Psi$. A permutation group is \emph{semi-regular} if only the identity has a fixed point. If it is additionally transitive, then it is \emph{regular}.

\paragraph{Restricted Group Classes}

Let~$\Gammad$ be the smallest class of groups satisfying the following properties:

\begin{enumerate}
\item the trivial group is in $\Gammad$,
\item $\Gammad$ is closed under taking extensions of subgroups of~$S_d$ (i.e., if~$\Psi\trianglelefteq \Delta$ with~$\Psi \in \Gammad$ and~$\Delta/\Psi \leq S_d$, then~$\Delta\in\Gammad$),
\item $\Gammad$ is closed under taking extensions of cyclic groups (i.e., if~$\Psi\trianglelefteq \Delta$ with~$\Psi \in \Gammad$ and~$\Delta/\Psi$ is cyclic, then~$\Delta\in\Gammad$).
\end{enumerate}

For convenience, in this paper within $\Gammad$ we explicitly allow extensions of abelian infinite groups by $\Gammad$-groups rather than only finite groups. This simplifies some of our arguments that deal with infinite graphs.
Another way of describing finite groups in~$\Gammad$ is that they are groups whose non-abelian composition factors are subgroups of~$S_d$.

Let~$\Thetad$ be the smallest class of groups satisfying the following properties:

\begin{enumerate}
\item the trivial group is in $\Thetad$,
\item $\Thetad$ is closed under taking direct products (i.e., if $\Delta,\Delta'\in \Thetad$,
    then $\Delta\times\Delta'\in \Thetad$)
\item $\Thetad$ is closed under taking wreath products with
  symmetric groups as top group (i.e., if~$\Delta \in \Thetad$ and $t\in\NN$ is arbitrary,
  then~$\Delta \wr S_t\in \Thetad$),
\item $\Thetad$ is closed under taking extensions of groups in~$\Gammad$ (i.e., if~$\Psi\trianglelefteq \Delta$ with~$\Psi \in\Thetad$ and~$\Delta/\Psi \in \Gammad$, then~$\Delta\in\Thetad$).
\end{enumerate}

Thus, the non-abelian composition factors of groups in~$\Thetad$ are subgroups of~$S_d$ or alternating groups. However, there is a restriction on how alternating groups may appear. Intuitively, alternating groups that are not subgroups of~$S_d$ can only arise if the respective symmetric group is present.

\paragraph{Rooted pointwise limits}

In this section, we recapitulate the limit constructions for rooted graphs described in~\cite{DBLP:journals/jgt/Babai91}. This also gives us the chance to verify that they apply not only to limits of vertex-transitive graphs but also to limits of edge-transitive graphs.

Let~$U$ be a set. A sequence~$X=(X_i)_{i\in\NN}$ of subsets of~$U$ converges
pointwise if for each~$x\in U$ there is an~$n_0$ so that for all~$n\geq n_0$ we have~$x\in X_n$
if and only if~$x \in X_{n_0}$. The limit
of~$X$ is the set of those~$x$ for which~$x\in X_n$ for all but
finitely many~$n$.

We will consider connected ordered graphs~$G$, that is, graphs whose vertex set is precisely the set~$\{1,\ldots,n\}$ for some integer~$n$. The vertex~$1$ can also be thought of as the \emph{root} of the graph. Following~\cite{DBLP:journals/jgt/Babai91}, we require that the vertices are ordered according to a breadth-first search traversal (\emph{BFS-labeled}), which means that for vertices~$i,j$ with~$i<j$ we have~$\dist_G(1,i)\leq \dist_G(1,j)$. 

Let~$(G_i)_{i\in\NN}$ be a sequence of finite BFS-labeled graphs. The graph~$G$ is the pointwise limit of the sequence if~$V(G)$ and~$E(G)$ are the pointwise limits of~$(V(G_i))_{i\in\NN}$ and~$(E(G_i))_{i\in\NN}$, respectively.

It follows from the
compactness principle (or directly from Tychonoff's Theorem) that
every sequence of graphs whose vertex set is a subset of the
natural numbers has a convergent subsequence.

\begin{lem}\label{lem:limitGraphs}
Let~$(G_j)_{j\in\NN}$ be a convergent sequence of connected finite, BFS-labeled, edge-transitive graphs of maximum degree at most~$d$. Let~$\limgraphG$ be the limit graph. Then
\begin{enumerate}
\item for every~$t\in \NN$ there is an~$n_t$ so that for all~$n\geq n_t$ the following holds: for every vertex~$v\in V(G_n)$ there is an isomorphism from $G_n[B_{t,G_n}(v)]$ to $\limgraphG[B_{t,\limgraphG}(x)]$ mapping~$v$ to~$x$ for some~$x\in \{1,2\}$.\label{convergence:lemma:part:balls}
\item $\limgraphG$ is connected and edge-transitive.\label{convergence:lemma:part:edge:transitive} 
\end{enumerate}
\end{lem}

\begin{proof}
We can assume without loss of generality that no graph~$G_n$ is
edgeless. This implies that~$\{1,2\}$ is an edge of each graph~$G_n$,
and thus of~$\limgraphG$.

We prove Part~\ref{convergence:lemma:part:balls}.
It follows from the limit construction and the fact that the graphs
are BFS-labeled that for each~$x\in \{1,2\}$ the subgraph of $G_n$ induced by ball~$B_{t,G_n}(x)$
converges to the subgraph of $\limgraphG$ induced by ball~$B_{t,\limgraphG}(x)$.
The statement now follows from edge-transitivity of the graphs~$G_n$.

We prove Part~\ref{convergence:lemma:part:edge:transitive}. Since the graphs~$G_j,j\in\NN$ are BFS-labeled, the limit graph~$\limgraphG$ is connected. Let~$e$ be
an edge of~$\limgraphG$. We show that there is an automorphism
mapping~$\{1,2\}$ to~$e$. For some~$n_0$ we have that for~$n\geq n_0$
the edge~$e$ also appears in~$G_n$.  For each~$n\geq n_0$ choose an
automorphism~$\varphi_n$ of $G_n$ mapping~$\{1,2\}$ to~$e$. Since the maximum
degree in~$G_n$ is bounded and the graphs are BFS-labeled, for
each~$v$, we can give a bound~$f(v)\in \NN$ independent of~$n$ so
that~$\varphi_n(v)\leq f(v)$. This implies that the sequence~$(\phi_n)_{n\geq n_0}$ has a
subsequence that converges to a function~$\varphi$.
The limit~$\varphi$ of this subsequence is an
automorphism of~$\limgraphG$ mapping~$\{1,2\}$ to~$e$ since it induces locally an
automorphism on the ball~$B_{t,G_{n_{g(t)}}}(1)$ for~$t$ arbitrarily large and some increasing function~$g$.
\end{proof}

As Babai does in~\cite{DBLP:journals/jgt/Babai91}, we also remark that the limit
construction can also be explained in terms of ultraproducts and
\L{}o\'s's Theorem (see, e.g., \cite[Chapter~9]{hod93}).

\section{Infinite edge-transitive graphs with two ends}\label{sec:twoEnds}

Following Babai's approach to analyze vertex-transitive graphs~\cite{DBLP:journals/jgt/Babai91}, we investigate the structure of large edge-transitive graphs by considering infinite limit graphs. We then draw conclusions about finite graphs from the properties of the infinite graphs.
Being almost vertex-transitive, the limit graph has one, two or infinitely many ends. These cases can be studied separately.
In~\cite{DBLP:journals/jgt/Babai91}, the most interesting case is that of one end and most results transfer fairly directly to the edge-transitive case. However, for edge-transitivity, the two-ended case turns out to be significantly more interesting and involved, so we study this case separately and first.
Throughout the section, we suppose~$G$ is a connected, locally finite, edge-transitive graph with two ends.
The \emph{connectivity between the two ends} is the minimum number of vertices in a separator separating the two ends.

\begin{lem}[Halin~\cite{MR270953}]\label{lem:paths}
If an infinite locally finite graph with two ends has connectivity~$k$ between the two ends,
then there are~$k$ vertex-disjoint bidirectionally infinite paths connecting the two ends.
\end{lem}

A connected graph $G$
is called a \emph{strip}
if there exists a connected set $C\subseteq V(G)$
and an automorphism $\psi\in\aut(G)$
such that $S\coloneqq N(C)$ is non-empty and finite,
$\psi(C\cup S)\subseteq C$,
and $C\setminus\psi(C)$ is finite.
In \cite{DBLP:journals/combinatorica/Jung81}, it is shown
that connected, locally finite, vertex-transitive graphs with two ends
are strips.
This also holds for edge-transitive instead of vertex-transitive graphs.

\begin{lem}\label{lem:translation}
Let $G$ be a connected, locally finite and edge-transitive graph with two ends.
Then, $G$ is a strip.
\end{lem}

\begin{proof}
We follow an idea similar to the proof of \cite[Theorem 1]{DBLP:journals/combinatorica/Jung81} for vertex-transitive graphs.
Let $S$ be a minimum separator separating the two ends.
Let $S^L$ and $S^R$ be the two infinite connected components of $G-S$
(containing the ends).

Let $D\geq 0$ be the diameter of $S$, i.e., the maximum distance of vertices $v,v'\in S$ measured in~$G$.
Since $G$ is edge-transitive and locally finite, we can pick a vertex $v\in S$ and an automorphism
$\psi\in\aut(G)$ such that $v^{\psi}\in S^L$ and $\dist(v,v^\psi)>2D$.
Then, it holds that $S\cap S^\psi=\emptyset$.
Similarly, we can pick a second automorphism
$\psi'$ such that $v^{\psi'}\in S^R$ and $\dist(v,v^{\psi'})>2D$.
\begin{claim}
There is an automorphism $\phi\in\{\psi,\psi',\psi\psi'\}$
such that $(C\cup S)^\phi\subseteq C$ for some $C\in\{S^L,S^R\}$.
\end{claim}
\begin{claimproof}
If $\psi$ or $\psi'$ does not interchange the two ends, then we are done since $S\cap S^\phi=\emptyset$
for both $\phi\in\{\psi,\psi'\}$.
Thus, assume that both $\psi,\psi'$ interchange the two ends,
i.e., $(S^R\cup S)^\psi\subseteq S^L$ and $(S^L\cup S)^{\psi'}\subseteq S^R$.
In this case, it holds that $(S^R\cup S)^{\psi\psi'}\subseteq (S^L)^{\psi'}\subseteq (S^L\cup S)^{\psi'}\subseteq S^R$.
\end{claimproof}

Note that $N(C)$ is finite
since $N(C)=S$ is a minimum finite separator,
and note that $C\setminus\psi(C)$ is finite since $G$ has two ends.
Thus, the claim completes the proof of the lemma.
\end{proof}

For two finite separators~$S$ and~$S'$ separating the two ends, we let~$[S,S']$ be the set of vertices of~$G$ that do not lie in an infinite connected component of~$G-(S\cup S')$.

\begin{lem}\label{lem:boundedDiameter}
Let $G$ be a connected, locally finite and edge-transitive graph with two ends.
There is a constant~$D$ that bounds the diameter of every minimum separator $S$ separating the two ends (i.e.,~$\dist_G(v,v')\leq D$ for all~$v,v'\in S$).
\end{lem}

\begin{proof}
Fix some minimum cardinality separator $S_0$ separating
the two ends.
By \cref{lem:translation}, there is
an automorphism $\psi\in\aut(G)$ of infinite order.
By possibly replacing $\psi$ with a suitable power, we can assume that
$[S_0,S_0^\psi]$ is a subset that is connected (using only paths inside $[S_0,S_0^\psi]$).
Now, consider the intervals $\ldots,I^{-1},I^0,I^1,I^2,\ldots$
where $I^\ell\coloneqq[S_0^{\psi^\ell},S_0^{\psi^{\ell+1}}]$ for $\ell\in\ZZ$.
By \cref{lem:paths}, there are $k\coloneqq|S|$ vertex-disjoint paths connecting
the two ends. Note that each minimum separator $S$ contains exactly one vertex from each path.
Thus, since $[S_0,S_0^\psi]$ is connected, for each minimum separator $S$ the set of intervals
$\{I^\ell\mid I^\ell\cap S\neq\emptyset\}$ that non-trivially intersect~$S$ must be a contiguous sequence of intervals.
Therefore,
each separator $S$ is contained in the interval $[S_0^{\psi^{k_0}},S_0^{\psi^{k_0+2k}}]$
for some $k_0\in\ZZ$ where $k=|S_0|=|S|$ is the connectivity between the two ends.
Note that the diameter of $[S_0^{\psi^{k_0}},S_0^{\psi^{k_0+2k}}]$ in $G$
is equal to the diameter $D\in\NN$ between vertices in $[S_0,S_0^{\psi^{2k}}]$ since $\psi^{k_0}$
is an automorphism of $G$.
Thus, the diameter of $S$ in $G$ is bounded by the constant $D$
(which only depends on $S_0$ and $\psi$ but not on $S$).
\end{proof}

For a generalization of this lemma see also \cite[Proposition 4.2 and Corollary 4.3]{DBLP:journals/jct/ThomassenW93}.

\paragraph{Level Sets}
Since~$G$ is edge-transitive, it has at most two vertex orbits. Let~$V_{\LeftMostSep}\subseteq V(G)$ be the union of all minimum separators of~$G$ separating the two ends.
If some vertex is in~$V_{\LeftMostSep}$,
then its entire vertex orbit is in~$V_{\LeftMostSep}$.
Thus, the set $V_{\LeftMostSep}$ contains an entire orbit $O_1$,
i.e., $O_1\subseteq V_{\LeftMostSep}$.

Fix the two ends of the graph~$G$, call them left and right. For every finite separator~$S$ that separates the two ends let~$S^L$ and~$S^R$ be the connected components of~$G-S$ containing the left and right end, respectively. 
Note that $N_G(S^L)=S=N_G(S^R)$ is~$S$ if a minimum separator.

\begin{lem}
For every vertex~$v\in O_1$, there is a unique leftmost minimum separator~$S_v$ containing~$v$ in the following sense:
for every minimum separator~$S$ separating the two ends and containing~$v$ we have~$ (S_v)^L \subseteq S^L$.
\end{lem}

\begin{proof}
Let $\CS=\{S\mid S$ is a minimum separator separating the ends that contains $v\}$
and note that $\CS$ is finite
since each $S\in\CS$ is contained in a ball centered at $v$ with radius $D$ (\cref{lem:boundedDiameter}).
For $S,S'\in\CS$,
the intersection~$S^L\cap {S'}^L$ is infinite since it contains the left infinite component of $G-(S\cup S')$.
We now use a submodularity argument. Define~$S^\cap \coloneqq (S\cap S') \cup (S \cap {S'}^L) \cup (S^L \cap S')$.
We claim that~$S^\cap \in \mathcal{S}$ and~${S^\cap}^L \subseteq S^L\cap {S'}^L$. Indeed,~$S^\cap$ is a separator separating the ends since every path from the left end to the right end has to enter~$S$ and~$S'$ at the same time or~$S$ first or~$S'$ first. The separator~$S^\cap$ is of minimum cardinality since otherwise~$S^\cup \coloneqq (S\cap S') \cup (S \cap {S'}^R) \cup (S^R \cap S')$ would be a separator of cardinality smaller than~$|S|=|S'|$ (since~$|S^\cup|+|S^\cap|=|S|+|S'|$).

Therefore, there is a leftmost minimum separator $S_v\in\CS$ where
$S_v^L=\bigcap_{S\in\CS} S^L$ is 
the intersection of finitely many left components
and where $S_v=N_G(S_v^L)$.
\end{proof}

The lemma implies in particular that if~$u\in S_v$ then~$S_u^L\subseteq S_v^L$.

For vertices~$v,v'\in O_1$ define~$D(v,v')\coloneqq |S^L_v\setminus S^L_{v'}| - |S^L_{v'}\setminus S^L_{v}|$. Note that~$D(v,v')$ is finite since~$S^L_v\setminus S^L_{v'}$ and~$S^L_{v'}\setminus S^L_{v}$ are finite. 
Fix some arbitrary vertex~$v_0\in O_1$.
We define a linearly ordered partition of $O_1$ into so-called \emph{(primary) level sets}
by partitioning the vertices $v\in O_1$ according to the value $D(v,v_0)\in\ZZ$.
More precisely, the ordered partition $(L_i)_{i\in\ZZ}$ of $O_1$ into non-empty sets is defined
such that $D(v,v_0)<D(w,v_0)$ if and only if $v\in L_i,w\in L_j$ for $i,j\in\ZZ$ with $i<j$.
Without loss of generality, we can assume that the indices $i\in\ZZ$ are chosen such that $L_0=\{v\in O_1\mid D(v,v_0)=0\}$ (which is non-empty since $v_0\in L_0$).
We also set the \emph{level} of a vertex $v\in L_i$ as $\LevelSet(v)\coloneqq\LevelSet(v)_{v_0}\coloneqq i\in\ZZ$,
and thus $L_i=\{v\in O_1\mid\LevelSet(v)=i\}$.
In case that $G$ has exactly two orbits $O_1,O_2$, we assign
\emph{secondary level sets} by defining
$J_{i+\frac{1}{2}}\coloneq\{v\in O_2\mid N(v)\subseteq L_i\cup
L_{i+1}\}$.
We also set the \emph{level} of $v\in J_{i+\frac{1}{2}}$ as
$\LevelSet(v)\coloneqq i+\frac{1}{2}$.
The next lemma in particular shows (Part~\ref{level:partition}) that every vertex of $G$ has a well-defined
level.

\begin{lem}\label{lem:levelSets}
Let $G$ be a connected, locally finite and edge-transitive graph with two ends.
Suppose~$u,v,w,v_0,v_0'\in O_1$.
\begin{enumerate}
  \item\label{level:add} $D(u,v)+D(v,w)=D(u,w)$.
  \item\label{level:c0} 
  Setting~$c\coloneqq D(v_0,v_0')$, 
  for all vertices $x\in O_1$ we have~$\LevelSet(x)_{v_0'}=\LevelSet(x)_{v_0}+c$.
  \item\label{level:finite} The cardinality of the (primary) level sets $L_i,i\in\ZZ$ is finite.
  \item\label{level:invariant} The partition $\CP\coloneqq\{L_i\mid i\in\ZZ\}$ of $O_1$ into (primary) level sets is $\aut(G)$-invariant. Moreover automorphisms map consecutive level sets to consecutive level sets.
  \item\label{level:partition} If $G$ has exactly two orbits, then $\CQ\coloneqq\{J_{i+\frac{1}{2}}\mid i\in\ZZ\}$ is 
  a partition of $O_2$ and this partition is $\aut(G)$-invariant.
  \item\label{level:endpoints} If $G$ is vertex-transitive, then all edges $e\in E(G)$ have endpoints in $L_i$ and $L_{i+1}$
  for some $i\in\ZZ$.
  \item\label{level:biendpoints} If $G$ has exactly two orbits, then
  all edges $e\in E(G)$ have endpoints in $L_i$ and $J_{i'}$ for some $i\in\ZZ$ and $i'\in\{i-\frac{1}{2},i+\frac{1}{2}\}$. 
  \item\label{level:regular} If $G$ is vertex-transitive, then the
    graph $G[L_i,L_{i+1}]$
    is regular for all $i\in\ZZ$.
  \item\label{level:biregular} If $G$ has exactly two orbits, then the graphs $G[L_i,J_{i+\frac{1}{2}}],G[J_{i+\frac{1}{2}},L_{i+1}]$ are biregular for all $i\in\ZZ$.
\end{enumerate}
\end{lem}
\begin{proof}
We prove Part \ref{level:add}.
Let $S_u,S_v,S_w$ be leftmost minimum separators.
Since $G$ is connected and locally finite, it follows that for each finite separator $S$ separating the two ends
there are finitely many connected components in $G-S$,
and exactly two of them are infinite.
Let $S_u^L,S_v^L,S_w^L$ be the respective connected components of $G-S_u,G-S_v,G-S_w$ containing the left end.
From what we just observed the graph $G-(S_u\cup S_v\cup S_w)$ has exactly two infinite
connected components (containing the left and right end),
and the one containing the left end is in turn contained in the intersection $I\coloneqq S_u^L\cap S_v^L\cap S_w^L$.
Therefore, the sets $S_u^L\setminus I,S_v^L\setminus I,S_w^L\setminus I$ are finite.
This implies that $D(u,v)=|S_u^L\setminus I|-|S_v^L\setminus I|$.
Similarly, it holds that $D(v,w)=|S_v^L\setminus I|-|S_w^L\setminus I|$
and $D(u,w)=|S_u^L\setminus I|-|S_w^L\setminus I|$.
Therefore, it holds that $D(u,v)+D(v,w)=|S_u^L\setminus I|-|S_v^L\setminus I|+|S_v^L\setminus I|-|S_w^L\setminus I|=D(u,w)$.
This proves Part \ref{level:add}.

We prove Part \ref{level:c0}.
Indeed note that Part \ref{level:c0} follows from Part \ref{level:add} by
setting $c\coloneqq D(v_0,v_0')$. 

We prove Part \ref{level:finite}.
We consider a leftmost minimum separator $S_v$ for some vertex $v\in L_i$.
Since the distance of vertices within a minimum separator is bounded by a constant $D$ (\cref{lem:boundedDiameter}),
there are vertices $u,w$ and leftmost minimum separators $S_u,S_w$
such that $S_u^L\subsetneq S_v^L\subsetneq S_w^L$
where $S_w^L\setminus S_u^L$ is finite.
Moreover, by choosing~$u$ and~$w$ so that~$S_u$ and~$S_v$ respectively~$S_w$ and~$S_v$ are sufficiently far apart, it holds that $L_i\subseteq S_w^L\setminus S_u^L$, and thus $L_i$ is finite.

We prove Part \ref{level:invariant}.
By Part \ref{level:c0}, the partition $\CP$ into level sets does not depend on $v_0$.
We need to be careful that the partition $\CP$ does not change when we swap the roles of the left and the right end,
i.e., if we were to define the level sets with respect to the right rather than the left end
by considering the \emph{rightmost minimum separator} $\tilde S_v$ containing $v$ with an inclusion minimal component $\tilde S_v^R$. 
Since~$\aut(G)$ acts transitively on~$O_1$,
the size of $V(G)\setminus (S_v^L\cup\tilde S_v^R)$ is an
invariant for all vertices $v\in O_1$ (and is finite since $G$ has two ends). It follows that two vertices~$v,v'\in O_1$ are in the same (primary) level set if and only if they are
in the same (primary) level set if we define level sets with respect to the right rather than the left end: Indeed, every vertex~$v$ partitions the vertex set into three parts: a left part~$S_v^L$, a right part~$\tilde S_v^R$, and a middle part $V(G)\setminus (S_v^L\cup\tilde S_v^R)$. Our original definition measures the difference of vertices in the left parts, while the definition with rightmost minimum separators would measure the difference of vertices in the right parts. However, the middle parts have the same number of vertices. This shows that the property
of being in the same level set is preserved under automorphisms.

It is clear that automorphisms that do not interchange the ends map consecutive level sets to consecutive level sets. 
 It follows with the same counting arguments that also automorphisms that do interchange the ends map consecutive level sets to consecutive level sets. 

We prove Part \ref{level:partition}.
Let $v$ be a vertex in~$O_2$. We argue that $v$ has neighbors in exactly two (primary) levels sets.

If $v$ has neighbors in only one (primary) level set, then this is the case for all vertices in $O_2$,
and since $G$ is edge-transitive, there is no path in $G$ connecting two distinct (primary) level sets,
contradicting that $G$ is connected.

Next, we rule out the case that $v$ has neighbors in more than two (primary) level sets.
We say that an edge $e=vw_j$ with $v\in O_2,w_j\in L_j$ \emph{lies between other edges} if there are edges $vw_i,vw_k\in E(G)$ such that
$w_i\in L_i,w_k\in L_k$ and $i<j<k$.
If~$v$ has neighbors in more than two level sets, then there are edges lying between others.
However, there are also edges that do not lie in between other edges since $G$ is locally finite.
This contradicts the fact that~$G$ is edge-transitive.
Thus, every vertex $v\in O_2$ has neighbors in exactly two (primary) level sets.

Finally, we show that each vertex $v$ has neighbors in two consecutive (primary) level sets.
Let $v_i,v_j\in N(v)$ such that $v_i\in L_i,v_j\in L_j,i\neq j$.
Since~$O_2$ is an orbit, we can conclude that the difference $|i-j|$
is an invariant across all vertices from~$O_2$.
If this invariant was different from~$1$, then the graph would not be connected.
More precisely, if the invariant would be $c>1$, then the
set~$O_2\cup\bigcup_{i\in\ZZ} L_{c\cdot i}\subsetneq V(G)$ would contain a non-trivial connected component of~$G$, contradicting that~$G$ is connected.
This means that $\CQ$ forms a partition of $O_2$.

We prove Part \ref{level:endpoints}.
Since~$G$ is edge-transitive and level sets are blocks under automorphisms, for every edge~$e=vw\in E(G)$ the value of~$|\LevelSet(v)-\LevelSet(w)|$ is an invariant among all edges.
(The sign of $\LevelSet(v)-\LevelSet(w)$ depends on the choice of the left end.)
Again, if this invariant were $c>1$, then
the set~$\bigcup_{i\in\ZZ} L_{c\cdot i}\subsetneq V(G)$ would contain a non-trivial connected component of~$G$, contradicting that~$G$ is connected.

Part \ref{level:biendpoints} follows directly from Part \ref{level:partition}
since each vertex $v\in J_{i+\frac{1}{2}}$ has only neighbors in
$L_i$ and $L_{i+1}$ for all $i\in\ZZ$.

We prove Part \ref{level:regular}.
Suppose for the sake of contradiction that two vertices in $L_i$ have a different number of neighbors in $L_{i+1}$.
(In principle, we might think this could happen because~$G$ has reflections swapping the two ends.)
Then, every level set~$L_i$ can be partitioned into two non-empty sets $L_i^>$ and $L_i^<$ ~consisting of those vertices~that have more neighbors in~$L_{i+1}$ than in~$L_{i-1}$ and vice versa.
Let $L^{>>}\coloneqq\bigcup_{i\in\ZZ} (L_i^>),L^{<<}\coloneqq\bigcup_{i\in\ZZ} (L_i^<)$ and
let $L^{><}\coloneqq\bigcup_{i\in\ZZ} (L_{2i}^>\cup L_{2i+1}^<),L^{<>}\coloneqq\bigcup_{i\in\ZZ} (L_{2i}^<\cup L_{2i+1}^>)$.
Note that $\CB_1\coloneqq\{L^{>>}, L^{<<}\}$ and
$\CB_2\coloneqq\{L^{><},L^{<>}\}$ are both block systems.
Let $e\in E(G)$ be an edge. Then, $e$ is contained in some block of one of those block systems, i.e.,
$e\subseteq B$ for some $B\in\CB_{i^*}$ and some $i^*\in\{1,2\}$.
By edge-transitivity, all edges are contained in blocks of $\CB_{i^*}$,
i.e., for all edges $e\in E(G)$ there is a block $B\in\CB_{i^*}$
such that $e\subseteq B$.
But then, there are no edges connecting the two blocks in $\CB_{i^*}$,
contradicting that $G$ is connected.
Therefore, each graph $G[L_i,L_{i+1}]$ is biregular. Since~$|L_i|=|L_{i+1}|$, the graph is even regular.

We prove Part \ref{level:biregular}.
We want to rule out that vertices in $L_i$ have a different number of neighbors in $J_{i+\frac{1}{2}}$.
If there is a vertex in $L_i$ that does not have a neighbor in
both $J_{i-\frac{1}{2}}$ and $J_{i+\frac{1}{2}}$, then this would hold
for all vertices in $L_i$ since $L_i\subseteq O_1$.
In this case, $J_{i-\frac{1}{2}}$ and $J_{i+\frac{1}{2}}$ would not be in the same connected
component, contradicting that $G$ is connected.
Thus, let $v_1,v_2\in L_i$ and
let $e_k$ be an edge connecting $v_k\in L_i$ with $J_{i+\frac{1}{2}}$ for $k=1,2$.
Since $G$ is edge-transitive, there is an automorphism
that maps $e_1$ to $e_2$. Since $L_i\subseteq O_1$ and $J_{i+\frac{1}{2}}\subseteq O_2$,
this automorphism also maps $v_1$ to $v_2$ and stabilizes $J_{i+\frac{1}{2}}$ setwise.
Therefore, the number $d_k^+\coloneqq|N_G(v_k)\cap J_{i+\frac{1}{2}}|$ of neighbors
is the same for both $k=1,2$.
Since $v_1,v_2\in L_i$ have the same degree $d_k(v_k)\coloneqq\deg_G(v_k)$, also the number
$d_k^-\coloneqq|N_G(v_k)\cap J_{i-\frac{1}{2}}|$
of neighbors coincide.
By swapping the role of the primary and secondary level sets,
the same arguments can be applied to vertices $w_1,w_2$ in a secondary level set~$J_{i+\frac{1}{2}}$.
Thus, the bipartite graph induced on two level sets is biregular.
\end{proof}

Now, we consider biregular graphs.

\begin{lem}\label{lem:doubleMatching}
Let $G$ be a graph and $A,B,C\subseteq V(G)$ disjoint subsets of the vertices such that $G[A,B]$ and~$G[B,C]$ are induced biregular graphs that are not edgeless. Suppose~$m\coloneqq|A|=|C|\leq|B|$.
Then, there are $m$ vertex-disjoint paths from $A$ to $C$.
\end{lem}

\begin{proof}
Let $S$ be a minimum separator between $A$ and $C$.
By Menger's theorem, it suffices to show that $|S|\geq m$.
Let $S_A\coloneqq S\cap A,S_B\coloneqq S\cap B$ and $S_C\coloneqq S\cap C$.
Let $R\subseteq V(G)$ be the set of vertices that can be reached by a path in $G-S$ starting in $A\setminus S$.
Let $R_A\coloneqq R\cap A,R_B\coloneqq R\cap B$ and $R_C\coloneqq R\cap C$.
Clearly, it holds that $|R_A|=|A|-|S_A|$.
Define $c\coloneqq \frac{|B|}{m}\geq 1$.
Since $G[A,B]$ is regular, the size of the neighborhood of $R_A\subseteq A$ in $B$ is at least $c\cdot|R_A|$.
Note that $N(R_A)\cap B\subseteq R_B\cup S_B$.
This leads to $|R_B|\geq |N(R_A)\cap B|-|S_B|\geq c\cdot |R_A|-|S_B|$.
With the same argument $|R_C|\geq |N(R_B)\cap C|-|S_C|\geq c^{-1}\cdot |R_B|-|S_C|$.
In total, we have that $|R_C|\geq|A|-|S_A|-c^{-1}\cdot|S_B|-|S_C|\geq |A|-|S|$.
On the other hand, $|R_C|=0$ since $S$ separates $A$ and $C$.
This means that $|S|\geq |A|=m$.
\end{proof}

The following lemma shows that the size of the (primary) level sets is the connectivity between the two ends.

\begin{lem}\label{lem:levelSetPaths}
Let $G$ be a connected, locally finite and edge-transitive graph with two ends
and let $m\coloneqq|L_0|$ be the size of a (primary) level set. There are $m$ vertex-disjoint paths
connecting the two ends, and thus $m$ is the connectivity between the two ends.
Furthermore, if $G$ has exactly two orbits, then $|L_0|\leq|J_{\frac{1}{2}}|$.
\end{lem}

\begin{proof}
\icase{1. $G$ is vertex-transitive}
By Part \ref{level:endpoints} and \ref{level:regular} of \cref{lem:levelSets}, each edge of $G$ lies within a regular graph $G[L_i,L_{i+1}]$ for some $i\in\ZZ$.
It follows from Hall's marriage theorem
that $G[L_i,L_{i+1}]$ has a matching of size $m\coloneqq|L_i|=|L_{i+1}|$.
This leads to $m$ vertex-disjoint paths $P_1,\ldots, P_m$ connecting the ends.

\icase{2. $G$ is not vertex-transitive} In that case~$G$ has a second orbit $O_2\neq O_1$.
By Parts \ref{level:biendpoints} and~\ref{level:biregular} of \cref{lem:levelSets},
each edge of $G$ belongs to one of the biregular graphs
$G[L_i,J_{i+\frac{1}{2}}],G[J_{i+\frac{1}{2}},L_{i+1}]$ for some $i\in\ZZ$.
Assume first that $m\leq|J_{i+\frac{1}{2}}|$.
By \cref{lem:doubleMatching}, there are $m$ vertex-disjoint paths from $L_i$ to $L_{i+1}$ (via $J_{i+\frac{1}{2}}$).
This leads to $m$ vertex-disjoint paths connecting the two ends.

Now, assume for the sake of contradiction that $|J_{i+\frac{1}{2}}|<m$.
Then, with the same argument, we obtain vertex-disjoint paths from $J_{i-\frac{1}{2}}$ to $J_{i+\frac{1}{2}}$.
But then, there are vertices in $O_1$ that are not contained in these paths
(and thus not contained in any minimum separator separating the two ends),
contradicting the fact that each vertex in $O_1\subseteq V_{\LeftMostSep}$ is in such a separator.
\end{proof}

Suppose that the connectivity between the two ends is~$m$. For bidirectionally infinite paths $P_1,\ldots,P_m$
and a subset $S\subseteq V(G)$,
we let $H^S_{P_1,\ldots,P_m}$
be the graph with vertex set $\{1,\ldots,m\}$ in which two vertices~$i,j\in[m]$ are adjacent if there is a path in~$G$
from some vertex in~$P_i$ to some vertex in~$P_j$
whose vertices are in $S$ and
whose internal vertices do not lie on any of the~$m$ paths~$P_1,\ldots,P_m$.
We also write $H_{P_1,\ldots,P_m}$ for $H^{V(G)}_{P_1,\ldots,P_m}$.

\begin{lem}\label{lem:paths:invariant}
Let $G$ be a connected, locally finite and edge-transitive graph with two ends
and let $m$ be the connectivity between the two ends.
For vertex-disjoint paths~$P_1,\ldots,P_m$ connecting the two ends there are vertex-disjoint paths~$P'_1,\ldots,P'_m$
connecting the two ends and an automorphism~$\psi\in \aut(G)$ of infinite order
which maps each path~$P'_i$ to itself so that~$H_{P_1,\ldots,P_m}= H_{P'_1,\ldots,P'_m}$.
\end{lem}

\begin{proof}
Let~$S$ be a minimum separator separating the two ends.
Then, the separator $S$ contains exactly one vertex from every path $P_i$.
By \cref{lem:translation},
there is an automorphism $\psi\in\aut(G)$ of infinite order
such that $S\cap S^{(\psi^t)}=\emptyset$
for all $t\in\ZZ,t\neq 0$.
Let~$P_i[S,S^\psi]$ be the restriction of~$P_i$ to the subset~$[S,S^\psi]\subseteq V(G)$. Note that~$P_i[S,S^\psi]$ is a path with one endpoint in~$S$ and one endpoint in~$S^\psi$.
By possibly replacing~$\psi$ with a suitable power of itself, we can ensure that the connections between the different paths that are responsible for edges in~$H_{P_1,\ldots,P_m}$ also occur in the interval~$[S,S^\psi]$, i.e., that
$H_{P_1,\ldots,P_m}=H^{[S,S^\psi]}_{P_1,\ldots,P_m}$.
Again, by possibly replacing~$\psi$ with a suitable power of itself, we can ensure that for all~$i\in [m]$ and all vertices~$v\in V(G)$ it holds that~$v\in P_i\cap S$ if and only if~$v^\psi \in P_i\cap S^\psi$.
Define~$P'_i$ to be~$\bigcup_{t\in\ZZ}P_i[S,S^\psi]^{(\psi^t)}$ for each $i\in[m]$. Then, the collection of paths~$P'_1,\ldots,P'_m$ together with~$\psi$ satisfies the requirements of the lemma.
\end{proof}

We write $P_\infty$ to denote be the bidirectionally infinite path on vertex set~$\mathbb{Z}$.

\begin{cor}\label{cor:HisMinor}
Let $G$ be a connected, locally finite and edge-transitive graph with two ends. Let $m$ be the connectivity between the two ends.
Let $P_1,\ldots,P_m$ be vertex-disjoint paths connecting the two ends
and let $H\coloneqq H_{P_1,\ldots,P_m}$.
Then, the Cartesian product~$H\square P_{\infty}$ is a minor of~$G$.
\end{cor}

\begin{proof}
By \cref{lem:paths:invariant}, we can assume that there is
an automorphism~$\psi$ of infinite order leaving the paths~$P_1,\ldots,P_m$ invariant,
and thus each edge
$\{i,j\}$ in~$H_{P_1,\ldots,P_m}$ is realized by
infinitely many disjoint connections between the corresponding
paths~$P_i$ and~$P_j$ in~$G$.
\end{proof}

The next lemma shows that we can find paths such that $H_{P_1,\ldots,P_m}$
is vertex-transitive.

\begin{lem}\label{lem:pathsTransitive}
Let $G$ be a connected, locally finite and edge-transitive graph with two ends
and let $m\coloneqq|L_0|$ be the size of a (primary) level set.
There are vertex-disjoint paths~$P_1,\ldots,P_m$ connecting the two ends such that $H_{P_1,\ldots,P_m}$ is vertex-transitive.
\end{lem}

\begin{proof}
By \cref{lem:levelSetPaths}, there are vertex-disjoint paths~$P_1,\ldots,P_m$ connecting the two ends,
and by Lemma~\ref{lem:paths:invariant}, we can assume that there is an automorphism~$\psi$
of infinite order mapping each path to itself.
We argue that if~$H\coloneqq H_{P_1,\ldots,P_m}$ is not vertex-transitive, we can choose different paths~$P'_1,\ldots,P'_m$ so that~$H'\coloneqq H_{P'_1,\ldots,P'_m}$ has more edges than~$H$. Since the $m$-vertex graph $H$ has at most $\binom{m}{2}$ edges, this eventually proves the lemma.
Suppose that~$H$ is not vertex-transitive, and thus there is no automorphism from~$i$ to~$j$ for some $i,j\in[m]$.
Let $L_0$ be a (primary) level set.
Since $L_0\subseteq O_1$, there is an automorphism $\phi\in\aut(G)$ that
maps the vertex of~$L_0$ belonging to $P_i$ to the vertex of~$L_0$ belonging to $P_j$.
Since the partition of $O_1$ into (primary) level sets is invariant under automorphisms (\cref{lem:levelSets} Part \ref{level:invariant}),
the automorphism $\phi$ stabilizes $L_0$ setwise.
Therefore, $\phi$ induces a permutation $\tilde\phi$ of $[m]$ (the indices of the paths) that maps $i$ to $j$.
Note that $P_1^\phi,\ldots,P_m^\phi$ are $m$ vertex-disjoint paths
that are invariant under the automorphism $\phi^{-1}\psi\phi$ of infinite order,
and also note that $H_{P_1^\phi,\ldots,P_m^\phi}=H^{\tilde\phi}$.

There are exactly two infinite connected components $L,R$ of $G-L_0$.
Let $P'_1,\ldots,P'_m$ be the paths that agree with $P_1,\ldots,P_m$ on $L\cup L_0$
and that agree with $P_1^\phi,\ldots,P_m^\phi$ on $L_0\cup R$
and define $H'\coloneqq H_{P_1',\ldots,P_m'}$.
It holds that $H'=H^{L\cup L_0}_{P_1,\ldots,P_m}\cup H^{L_0\cup R}_{P_1^\phi,\ldots,P_m^\phi}$.

Since $P_1,\ldots,P_m$ are invariant under some automorphism of infinite order,
every edge of $H$ is supported infinitely many times inside of~$L$, and thus $H\subseteq H'$.
Similarly (with $R$ in the role of $L$), the paths $P_1^\phi,\ldots,P_m^\phi$ are invariant under some automorphism of infinite order,
and thus $H^{\tilde\phi}\subseteq H'$.
However, the graphs $H$ and $H^{\tilde\phi}$ are not identical,
otherwise the permutation $\tilde\phi$ of $[m]$ would
be an automorphism of $H$ mapping $i$ to $j$.
Therefore, the graph $H'$ is a proper supergraph of $H$.
\end{proof}

Thus, we can find paths such that $H_{P_1,\ldots,P_m}$ is vertex-transitive.
Note that finite connected vertex-transitive graphs with at least three vertices are 2-connected.
We make a case distinction of $H$ being a cycle or not,
and each case is handled separately in one of the following two lemmas.

\begin{lem}\label{lem:two:ends:H:cylindrical:or:small}
Let~$H$ be a finite 2-connected graph with $n\geq 3$ vertices that is not a cycle. Then, the graph $H\square P_{\infty}$ has a~$K_{n-1}$-minor.
\end{lem}

\begin{proof}
We construct a minor with~$n-1$ branch sets $B_1,\ldots,B_{n-1}\subseteq V(H\square P_{\infty})$ as follows.
Each slice~$S_i\coloneqq V(H) \times \{i\}\subseteq V(H\square P_{\infty})$ will contain at least one vertex from each branch set, i.e.,
$S_i\cap B_j\neq\emptyset$.
Exactly one branch set will intersect the slice in two vertices, and these two vertices form an edge, i.e., for each $i$ there
is exactly one $B_j$ such that $e_i\coloneqq S_i\cap B_j$ is a 2-element subset, and for this set it holds that $e_i\in E(H\square P_{\infty})$.
In the next slice~$S_{i+1}$ the same vertices of~$H$ will intersect the same
branch sets except that for one vertex~$v\in V(H)$ the branch set will be different in~$S_i$ compared
to $S_{i+1}$, i.e.,~$(v,i)$ and~$(v,i+1)$ are in distinct branch sets for exactly one~$v\in V(H)$.
In each of the slices this vertex has to be the vertex which appears in the branch set that contains two vertices,
i.e., $(v,i)\in e_i,(v,i+1)\in e_{i+1}$.
It suffices now to construct the branch sets so that for each pair of branch sets~$B,B'$
there is a slice~$S_i,i\in\ZZ$ so that~$B\cap S_i$ and~$B'\cap S_i$ are adjacent.

Overall, this translates into the following sliding puzzle. Consider the graph~$H$
and suppose there are~$n-1$ pebbles (corresponding to the branch sets) placed on~$n-1$ different vertices of the graph.
A legal move is to move one pebble across an edge to the previously unoccupied spot, called \emph{gap} in the following.
The edge across which a pebble slides in step~$i$ corresponds to the edge~$e_i$.
To solve the puzzle, the task is to perform a sequence of legal moves so that over time each pair of
pebbles was situated on adjacent vertices at some point.

We now argue that if~$H$ is 2-connected and not a cycle, then the puzzle is solvable. Let~$p$ and~$p'$ be non-adjacent pebbles on~$H$.
We first observe that we can perform a sequence of moves so that~$p$,~$p'$ and the gap lie on a common cycle. Indeed,~$p$ and~$p'$ lie on a common cycle~$C$ since $H$ is 2-connected. Due to 2-connectivity, there are two shortest paths~$P$ and~$P'$ from the gap to~$C$
which are vertex-disjoint (except on the gap).
If one of these paths does not end in~$p$ or~$p'$,
we can directly move the gap onto the cycle.
Otherwise, the paths $P$ and $P'$ end in~$p$ and~$p'$, respectively.
In that case, the paths~$P$ and~$P'$ together with a path in~$C$ joining~$p$ and~$p'$ form the desired cycle.

Let~$\tilde C$ be a cycle containing~$p$,~$p'$ and the gap.
Since~$H$ is not a cycle and due to 2-connectivity, there is a path~$\tilde P$ whose endpoints~$v,v'$ are distinct vertices on~$\tilde C$ and whose internal vertices are not on~$\tilde C$. We rotate the cycle~$\tilde C$ (by moving the gap on the cycle) so that~$p$ is located on~$v$, and we then move the gap along~$\tilde C$ without moving~$p$ so that the gap is located on~$v'$.
On~$\tilde C$ there are two paths~$\tilde P_1,\tilde P_2$ from~$v$ to~$v'$ and with each of them the path~$\tilde P$ forms a cycle.
One of these cycles~$\tilde P_1\cup \tilde P,\tilde P_2\cup \tilde P$ does not contain~$p'$. 
We may assume that $\tilde P_1$ does not contain $p'$.
We argue that we may assume that~$\tilde P$ has an internal vertex. Indeed, if~$\tilde P$ does not have an internal vertex, then~$\tilde P_1$ must have an internal vertex.
In that case, we interchange the names of~$\tilde P$ and~$\tilde P_1$, thereby replacing~$\tilde C$ with~$(\tilde C\setminus \tilde P_1) \cup \tilde P$.

We rotate the cycle~$\tilde P_1\cup \tilde P$ by one so that the
gap remains on~$\tilde C$ but the pebble $p$ that was on $v$
is now on the internal vertex of~$\tilde P$ adjacent to~$v$.
Finally, we rotate~$\tilde C$ to move~$p'$ to~$v$, making~$p$ and~$p'$ adjacent.
\end{proof}

\begin{cor}\label{cor:small}
Let $G$ be a connected, locally finite, $K_{h+1}$-minor-free and edge-transitive graph with two ends.
Let $m\coloneqq|L_0|$ be the size of a (primary) level set
and let $P_1,\ldots,P_m$ be
vertex-disjoint paths connecting the two ends.
If $H\coloneqq H_{P_1,\ldots,P_m}$ is 2-connected (with at least three vertices) and not a cycle, then $m\leq h+1$.
\end{cor}

\begin{proof}
By \cref{cor:HisMinor},
the Cartesian product $H_{P_1,\ldots, P_m}\square P_\infty$ is a minor of $G$, and
we conclude from Lemma~\ref{lem:two:ends:H:cylindrical:or:small} that~$m\leq h+1$.
\end{proof}

The \emph{twisted cylindrical grid of thickness $k$} is the infinite graph $G$
with vertex set $V(G)=\{(i,j)\mid i\in\ZZ, j\in[k]$\}
and edge set $E(G)=\{(i,j)(i+1,j')\mid i\in\ZZ, j'=j\text{ or }j'-1\equiv j \mod k\}$.
See \cref{fig:cylindrical}.

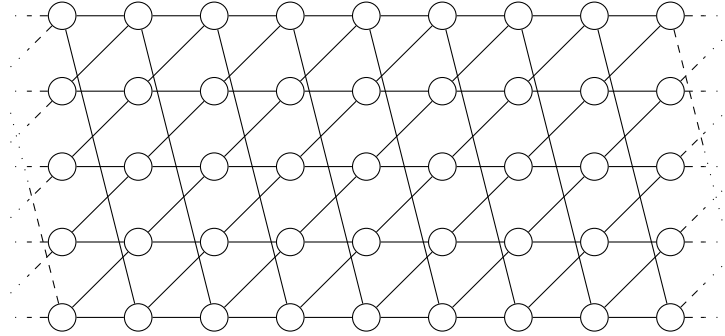
\begin{figure}[h]
\begin{center}
\begin{tikzpicture}
\foreach \i in {1,...,9}{
\foreach \j in {0,...,4}{

\node[draw,circle] (\i\j) at (\i,\j) {};

}}

\foreach \i in {0,10}{
\foreach \j in {0,...,4}{

\node (\i\j) at (\i,\j) {};

}}

\foreach \i [evaluate={\ii=int(\i+1)}] in {1,...,8}{
\foreach \j in {0,...,4}{
\foreach \x [evaluate={\jj=int(mod(\j+\x,5))}] in {0,1}{

\draw (\i\j) -- (\ii\jj);

}}}

\foreach \i\ii in {9/10}{
\foreach \j in {0,...,4}{
\foreach \x [evaluate={\jj=int(mod(\j+\x,5))}] in {0,1}{

\draw (\i\j) edge[dashed] ($(\i\j)!0.45!(\ii\jj)$);
\draw ($(\i\j)!0.45!(\ii\jj)$) edge[loosely dotted] ($(\i\j)!0.7!(\ii\jj)$);

}}}

\foreach \i\ii in {1/0}{
\foreach \j in {0,...,4}{
\foreach \x [evaluate={\jj=int(mod(\j-\x+5,5))}] in {0,1}{

\draw (\i\j) edge[dashed] ($(\i\j)!0.45!(\ii\jj)$);
\draw ($(\i\j)!0.45!(\ii\jj)$) edge[loosely dotted] ($(\i\j)!0.7!(\ii\jj)$);

}}}

\end{tikzpicture}
\end{center}
\caption{Twisted cylindrical grid of thickness 5.\label{fig:cylindrical}}
\end{figure}

\begin{lem}\label{lem:cylindrical}
Let $G$ be a connected, locally finite, $K_{h+1}$-minor-free, twin-free and edge-transitive graph with two ends.
Let $m\coloneqq|L_0|$ and let $P_1,\ldots,P_m$ be
vertex-disjoint paths connecting the two ends.
If $H\coloneqq H_{P_1,\ldots,P_m}$ is a cycle (with at least three vertices), then $m\leq h+1$ or $G$ is a subdivision of the twisted cylindrical grid.
\end{lem}

\begin{proof}
By possibly applying \cref{lem:paths:invariant},
we can assume that the paths $P_1,\ldots,P_m$ are invariant under some
automorphism of infinite order.
By possibly renaming the indices, we can assume that $H$ is the cycle $1,\ldots,m,1$.
Furthermore, we can assume that $m\geq 4$, otherwise if $m= 3$, then $h\geq 3$ and thus $m=3\leq h+1$.
\icase{1. $G$ is vertex-transitive}
Fix some $i\in\ZZ$ and
let $v_1,\ldots,v_m$ be the vertices of $P_1,\ldots,P_m$ in $L_i$
and $w_1,\ldots,w_m$ be the vertices of $P_1,\ldots,P_m$ in $L_{i+1}$.
Also, define $v_{m+1}\coloneqq v_1,w_0\coloneqq w_m,w_{m+1}\coloneqq w_1$.
We call a set of edges $X\subseteq E(G)$ is \emph{crossing} if
there is some $\ell\in[m]$ such that
$X\coloneqq\{v_{\ell}w_{\ell},v_{\ell}w_{\ell+1},v_{\ell+1}w_{\ell},v_{\ell+1}w_{\ell+1}\}$.
  
\begin{claim}
The graph $G$ is the twisted cylindrical grid,
or $G[L_i,L_{i+1}]$ has a crossing edge set $X$.
\end{claim}

\begin{claimproof}
Since $H$ is a cycle, the degrees in $G[L_i,L_{i+1}]$ are at least two
and at most three.
Consider the subgraph $H'\coloneqq H^{[L_i,L_{i+1}]}_{P_1,\ldots,P_m}\subseteq H$
and the directed graph $H_{\rightarrow}'$ where 
there is a directed edge $(\ell,\ell')$ for each edge $\{\ell,\ell'\}\in E(H')$
for which $v_\ell w_{\ell'}\in E(G)$.
(Note that an unordered pair $\{\ell,\ell'\}\in E(H'),\ell\neq\ell'$ has a directed edge in both directions
if and only if $X\coloneqq\{v_{\ell}w_{\ell},v_{\ell}w_{\ell'},v_{\ell'}w_{\ell},v_{\ell'}w_{\ell'}\}$
is a crossing edge set).
If $G[L_i,L_{i+1}]$ is $(2,2)$-biregular, then
the directed graph $H_{\rightarrow}'$ has only vertices with an indegree and outdegree of 1,
and thus it is a disjoint union of directed cycles.
If $H_{\rightarrow}'$ is a single directed cycle, then
$E(G[L_i,L_{i+1}])=\{v_\ell w_\ell\mid \ell\in[m]\}\cup\{v_\ell w_{\ell'}\mid\ell\in[m]\}$
where $\ell'\in\{\ell+1,\ell-1\}$,
and in particular, $G[L_i,L_{i+1}]$ and $G[L_{i+1},L_i]$ are isomorphic.
By using \cref{lem:levelSets}, we conclude that the graphs $G[L_k,L_{k+1}],k\in\ZZ$
are all pairwise isomorphic and capture all edges of $G$,
and thus $G$ is the twisted cylindrical grid.
If $H_{\rightarrow}'$ is a disjoint union of more than
one directed cycle, then these directed cycles must be of length 2 (two vertices with two directed edges)
since $H'$ is a subgraph of $H$.
But if $(\ell,\ell'),(\ell',\ell)$ both are directed edges in $H_{\rightarrow}'$, then
$X\coloneqq\{v_{\ell}w_{\ell},v_{\ell}w_{\ell'},v_{\ell'}w_{\ell},v_{\ell'}w_{\ell'}\}\subseteq E(G)$
is a crossing edge set.

In the remaining case,
we assume that $G[L_i,L_{i+1}]$ is $(3,3)$-regular,
and we show that there is a crossing edge set $X$.
Indeed, in that case for each~$\ell\in[m]$ the neighborhood of $v_\ell$
is precisely $\{w_{\ell-1},w_{\ell},w_{\ell+1}\}$.
Thus, the set $X\coloneqq\{v_{\ell}w_{\ell},v_{\ell}w_{\ell+1},v_{\ell+1}w_{\ell},v_{\ell+1}w_{\ell+1}\}\subseteq E(G)$
is a crossing edge set.
This proves the claim.
\end{claimproof}

In case $G[L_i,L_{i+1}]$ has a crossing edge set $X$ for some $i\in\ZZ$, we use
the crossing edge set $X$ to define new paths as follows. Formally, we
define paths $P_1',\ldots,P_m'$ as the symmetric difference
of the edge sets of $P_1,\ldots,P_m$ and the set of edges $X$.
These are basically the same paths except
that some end of two paths is swapped.
We consider the graph $H'\coloneqq H_{P_1',\ldots,P_m'}$.
Since the paths $P_1,\ldots,P_m$ are invariant under
some automorphism of infinite order, each edge in $H$ is supported infinitely many times.
Thus, if $L$ and~$R$ are the two infinite connected components of $G-(L_i\cup L_{i+1})$,
then $H^L_{P_1',\ldots,P_m'}$ and $H^R_{P_1',\ldots,P_m'}$ are both cycles.
However, these two cycles are not identical due to the swap of two paths (and since $m\geq 4$).
Therefore, the graph $H'$ is a proper supergraph of a cycle,
and in particular 2-connected.
By \cref{cor:HisMinor}, 
the graph $H'\square P_\infty$ is a minor of $G$,
and it follows from \cref{lem:two:ends:H:cylindrical:or:small} that $m\leq h+1$.

\icase{2. $G$ is not vertex-transitive}
In that case~$G$ has a second orbit $O_2\neq O_1$.

If $|J_{i-\frac{1}{2}}|=|L_i|$, then we can use the same arguments as in the vertex-transitive case
and conclude that $m\leq h+1$ or $G$ is a twisted cylindrical grid.
Otherwise, we have that $|J_{i-\frac{1}{2}}|> m$. If the vertices in $J_{i-\frac{1}{2}}$ have degree 2,
then we can dissolve them by deleting each such vertex and adding an edge between its two neighbors.
This gives us a vertex- and edge-transitive (topological) minor of $G$.
Therefore, also in this case, we conclude that $m\leq h+1$ or $G$
is a subdivision of the twisted cylindrical grid.

In the remaining case, we have that $|J_{i-\frac{1}{2}}|> m$ and the degree of vertices in $J_{i-\frac{1}{2}}$ within the graph $G[J_{i-\frac{1}{2}},L_i]$
is at least 2 (and at least~4 in~$G$). In the following, we argue that the degrees are exactly 2.
Let $v_1,\ldots,v_m$ be the vertices of $P_1,\ldots,P_m$ in $J_{i-\frac{1}{2}}$
and $w_1,\ldots,w_m$ be the vertices of $P_1,\ldots,P_m$ in $L_i$.
Also, define $v_{m+1}\coloneqq v_1,w_{m+1}\coloneqq w_1$.
Since $|J_{i-\frac{1}{2}}|> m$ there is a vertex $v\in J_{i-\frac{1}{2}}$
that is not contained in the paths $P_1,\ldots,P_m$.
Since $H$ is a cycle of length~$m\geq 4$, the neighborhood of the vertex $v$ in $L_i$ can only consist of two vertices $w_\ell,w_{\ell+1}$
for some $\ell\in[m]$.
Then, since $G[J_{i-\frac{1}{2}},L_i]$ is biregular (\cref{lem:levelSets}),
all vertices in $J_{i-\frac{1}{2}}$ have degree 2 in $G[J_{i-\frac{1}{2}},L_i]$ (and~4 in~$G$).

In the following, we will use that
$G$ is twin-free in order to find a crossing edge set (that will be defined similarly to the vertex-transitive case).
Consider the 2-element sets $N(v)\cap L_i$ for all vertices $v\in J_{i-\frac{1}{2}}$
(including the vertices that are contained in the fixed paths $P_1,\ldots,P_m$).
Since $H$ is a cycle, it holds that $N(v)\cap L_i=\{w_\ell w_{\ell+1}\}$ for some $\ell\in[m]$.
Since $|J_{i-\frac{1}{2}}|>|L_i|$, there are two vertices $v,v'$ having the same neighborhood in $L_i$,
i.e., there is an $\ell^*\in[m]$ such that $N(v)\cap L_i=\{w_{\ell^*},w_{\ell^*+1}\}=N(v')\cap L_i$.

Consider the case that none of the vertices $v,v'$
is contained in the paths $P_1,\ldots,P_m$.
Let $u_1,\ldots,u_m$ be the vertices of $P_1,\ldots,P_m$ in $L_{i-1}$, and set $u_{m+1}\coloneqq u_1$.
Then, it holds that $N(v)\cap
L_{i-1}=\{u_{\ell^*},u_{\ell^*+1}\}=N(v')\cap L_{i-1}$
(for the same $\ell^*\in[m]$ as above)
since if $u_k\in N(v)\cap L_{i-1}$ for $k\notin\{\ell^*,\ell^*+1\}$, then there would be
three paths $u_k,v,w_{\ell^*}$ and $u_k,v,w_{\ell^*+1}$ and $w_{\ell^*},v,w_{\ell^*+1}$,
contradicting that $H$ is a cycle of length~$m\geq 4$.
Thus, $v$ and $v'$ are twins, contradicting that $G$ is twin-free.

Consider the case that both vertices $v,v'$
are contained in the paths $P_1,\ldots,P_m$.
Since $v,v'$ have the same neighborhood in $L_i$,
there is some $\ell\in[m]$ such that $\{v,v'\}=\{v_\ell,v_{\ell+1}\}$.
Without loss of generality, we can assume that $v=v_\ell,v'=v_{\ell+1}$.
Then, we can find a crossing edge set
and swap the paths $P_\ell$ and $P_{\ell+1}$ as follows.
We define new paths $P_1',\ldots,P_m'$
by deleting the two edges $v_\ell w_\ell ,v_{\ell+1} w_{\ell+1}$
and adding the two edges $v_{\ell} w_{\ell+1},v_{\ell+1} w_\ell$.
With the same argument as in the vertex-transitive case,
the new paths $P_1',\ldots,P_m'$ lead to a new graph $H'$ that is 2-connected and not a cycle
such that $H'\square P_\infty$ is a minor of $G$.
We conclude from Lemma~\ref{lem:two:ends:H:cylindrical:or:small} that~$m\leq h+1$.

Finally, consider the case that the vertex $v$, but not $v'$, is contained in
the paths $P_1,\ldots,P_m$.
Without loss of generality assume that $v=v_\ell$
and that $N(v)\cap L_i=\{w_{\ell},w_{\ell+1}\}=N(v')\cap L_i$.
It holds that $N(v')\cap L_{i-1}=\{u_\ell,u_{\ell+1}\}$
(since $H$ is a cycle and $v'$ is not contained in the paths).
Clearly, it holds that $u_\ell\in N(v_\ell)\cap L_{i-1}$.
Again, we can swap the two paths.
We delete all edges in the two paths $u_\ell,v_\ell,w_\ell$ and $u_{\ell+1},v_{\ell+1},w_{\ell+1}$
and add the edges in the two paths $u_\ell,v_\ell,w_{\ell+1}$
and $u_{\ell+1}, v',w_\ell$.
\end{proof}

Overall, we obtain the following lemma concluding this section.

\begin{lem}\label{lem:comp:factors:of:two:end:vertex:trans}
Let $G$ be a connected, locally finite, $K_{h+1}$-minor-free, twin-free and edge-transitive graph with two ends.
Then, the automorphism group of~$G$ is a~$\Gamma_{h+1}$-group.
\end{lem}

\begin{proof}
By \cref{lem:pathsTransitive}, there are vertex-disjoint paths
$P_1,\ldots,P_m$ connecting the two ends
such that $H\coloneqq H_{P_1,\ldots,P_m}$
is vertex-transitive.
As a finite vertex-transitive graph, the graph~$H$ is 2-connected or has at most 2 vertices.
In the latter case, it holds that $|L_i|=1< h$ or $|L_i|=2<h$ for all (primary) levels $i\in\ZZ$.
In the former case, we apply
\cref{cor:small,lem:cylindrical} to conclude that
$G$ is a subdivision of the twisted cylindrical grid or
$|L_i|\leq h+1$ for all $i\in\ZZ$.

In either case, the automorphism group $\aut(G)$ has a normal subgroup~$\Delta\trianglelefteq\aut(G)$ that leaves the level sets fixed. The quotient $\aut(G)/\Delta$ is a cyclic or an infinite dihedral group.
If an automorphism in $\Delta$ fixes all points in the sets~$L_i$, it must fix all points since~$G$ is twin-free.
Since the action on the (primary) level sets is faithful, it thus suffices to consider the action of~$\Delta$ on the sets~$L_i$.
If~$|L_i|\leq h+1$ the normal subgroup~$\Delta$ is a subgroup of a direct product of symmetric groups~$S_{h+1}$.
If $G$ is a subdivision of the twisted cylindrical grid,
the normal subgroup $\Delta$ is a subgroup of a direct product of dihedral groups since~$H$ is a cycle.
\end{proof}

\section{Finite edge-transitive graphs}\label{sec:finite:edge:trans}

We now turn to connected finite edge-transitive graphs. Recall that these are regular or bipartite and semi-regular.
We will first investigate the possible degrees that may occur in $K_{h+1}$-minor-free graphs.

\begin{theo}[Kostochka~\cite{DBLP:journals/combinatorica/Kostochka84}]\label{theo:1}
  There is a constant $a\ge 1$ such that for every $h\ge 1$ the average
  degree of a finite $K_{h+1}$-minor-free graph is at most $a\cdot h\cdot\sqrt{\log h}$.
\end{theo}

In the following, we use $\alpha_h\coloneqq\lceil a\cdot h\cdot\sqrt{\log h}\rceil$
where $a$ is the constant of the theorem.
We say that a bipartite graph $G$ with bipartition $V_1,V_2$ is \emph{left-twin-free}
if there are no distinct vertices in $V_1$ that are twins.

\begin{lem}\label{lem:degree:bound}
  Let $G$ be a $(c_1,c_2)$-biregular, left-twin-free,
  $K_{h+1}$-minor-free, bipartite finite graph with bipartition $V_1,V_2$ such that $c_1\le c_2$. Then, it holds that
  $c_2\leq\alpha_h\cdot\left(\binom{\alpha_h}{\lceil\alpha_h/2\rceil}+1\right)$.
\end{lem}

\begin{proof}
The assertion is trivial for $h= 1$
and for $c_1\le 1$.
Also, note that $c_1\leq \alpha\coloneqq\alpha_h$ by \cref{theo:1}.
Let us
assume that $h\ge 2$ and $c_1\ge 2$. 
Suppose for the sake of contradiction that
    \begin{equation}
      \label{eq:3}
      c_2> \alpha\cdot\left(\binom{\alpha}{\lceil\alpha/2\rceil}+1\right).
    \end{equation}
    
We argue that there is a subset of edges~$E'\subseteq E(G)$ such that
each vertex in~$V_1$ is incident with at most one edge of~$E'$ and
each vertex of~$V_2$ is incident with at
least~$\binom{\alpha}{\lceil\alpha/2\rceil}+1$ edges
of~$E'$.
Indeed, create~$\binom{\alpha}{\lceil\alpha/2\rceil}$ copies of
each vertex of~$V_2$ giving us a new set~$V_2'$ in which each vertex
is in a twin class of size~$\binom{\alpha}{\lceil\alpha/2\rceil}+1$.
Since~$|V_2'|= \left(\binom{\alpha}{\lceil\alpha/2\rceil}+1\right)\cdot|V_2|<\frac{c_2}{c_1}\cdot |V_2|= |V_1|$,
by Hall's marriage
theorem there is a matching that matches each vertex in~$V_2'$ with
a vertex in~$V_1$. Identifying the twins again yields the desired
set of edges~$E'$.

Let $M$ be the minor of $G$ with vertex set $V(M)=V_2$ obtained from
$G$ by contracting the edges in~$E'$.
We show that $M$ has
average degree greater than $\alpha$.

Observe that for every $v_2\in V_2$ it holds that
  \[
    N_M(v_2)\supseteq\bigcup_{v_1v_2\in E'}(N_G(v_1)\setminus\{v_2\}).
  \]
In the following, we argue that $N_M(v_2)$ is large.
First, note that $|\{v_1v_2\in E'\}|>\binom{\alpha}{\lceil\alpha/2\rceil}\geq \binom{\alpha}{c_1-1}$.
Since $G$ is left-twin-free and all vertices in $V_1$ have degree $c_1$,
the sets $N_G(v_1)\setminus\{v_2\}$ are mutually distinct sets of size
$c_1-1$.
Therefore, we have that $\deg_M(v_2)=|N_M(v_2)|>\alpha$.
Thus, the average degree of $M$ is greater than $\alpha$.
By Theorem~\ref{theo:1}, the graph $M$ has a $K_{h+1}$-minor, contradicting that $G$ is
$K_{h+1}$-minor-free.
\end{proof}

Let us briefly record that it is not possible to have a subexponential bound in the previous lemma. Indeed, 
for each triple~$(t,h,r)$ of positive integers with~$r\leq h$ there is an edge-transitive graph of order~$t\binom{h}{r}^2+th$ that is
connected,
twin-free,
$\left(2r,2\binom{h-1}{r-1}\binom{h}{r}\right)$-biregular, and
$K_{\CO(h)}$-minor-free. 
For this, let~$V$ be the set~$\mathbb{Z}_t \times \{1,\ldots,h\}$.
Let~$U$ be the set~$\mathbb{Z}_t \times \binom{\{1,\ldots,h\}}{r} \times \binom{\{1,\ldots,h\}}{r}$.
Connect~$u= (i,A_1,A_2)\in U$ with~$v =(i,j)\in V$ if~$j\in A_1$ and also connect~$u= (i,A_1,A_2)$ with~$v =(i+1,j)$ if~$j\in A_2$.
For all expressions, the first indices are taken modulo~$t$.

Our goal in the rest of this section is to characterize the composition factors of edge-transitive graphs as follows.

\begin{theo}\label{thm:edge:transitive:twin:free:comp:fac:bound}
There is a function~$f$ such that every automorphism group of a connected $K_{h+1}$-minor-free, edge-transitive, twin-free and finite graph is contained in~$\Gamma_{f(h)}$.
\end{theo}

Towards the theorem,
assume for the sake of contradiction that there is an $h$ and an infinite sequence~$H_1,H_2,H_3,\ldots$ of connected $K_{h+1}$-minor-free, edge-transitive, twin-free finite graphs for which there is no~$d$ such that~$\aut(H_j)\in \Gammad$ for all~$j\geq 1$.
We can assume that if~$\aut(H_{j+1})\in \Gammad$, then~$\aut(H_{j})\in \Gammad$ for all $d\geq 0,j\geq 1$.
By \cref{theo:1,lem:degree:bound}, there is a constant bounding
the degree of each graph in the sequence.
As argued in the preliminaries,
this sequence has a convergent subsequence $G_1,G_2,G_3,\ldots$ and a corresponding connected edge-transitive infinite limit graph~$\limgraphG$.
Since the balls of $\limgraphG$ correspond to balls of graphs $G_j$ (see \cref{lem:limitGraphs}),
the limit graph $\limgraphG$ is also twin-free, $K_{h+1}$-minor-free and locally finite.

We perform a case distinction depending on the number of ends of~$\limgraphG$ and each possibility will give us a contradiction. 
As mentioned in the preliminaries,
the number of ends of an infinite, connected, almost vertex-transitive and locally finite graph is one, two or infinite.

\subsection{One end}\label{subsec:one:end} Suppose~$\limgraphG$ has one end. In this case,
we can apply various techniques that have been previously developed for vertex-transitive graphs.
However, we need to ensure that they apply to the edge-transitive case.
We first collect some information on the connectivity of~$\limgraphG$.

 \begin{lem}[Mader~\cite{MR289343}]\label{lem:mad}
   A finite, connected and edge-transitive graph of minimum degree $d$ has
     connectivity at least $d$.
 \end{lem}

A graph~$G$ is \emph{almost 4-connected} if it is 3-connected and for every 3-separator~$S$ the graph~$G-S$ has exactly two connected components, one of which consists only of one vertex.

\begin{lem}\label{lem:1:end:connectivity}
In case the limit graph $\limgraphG$ has only one end, it is almost-4-connected or~$\limgraphG$ has two orbits and the vertices in one of the orbits have degree~$2$. In the latter case,~$\limgraphG$ is a subdivision of a vertex-transitive and edge-transitive graph that has the same automorphism group as~$\limgraphG$.
\end{lem}

\begin{proof}
Let~$S$ be a minimum separator.
Since $\limgraphG$ is locally finite, the separator $S$ is finite.
Note that exactly one of the connected components of~$\limgraphG-S$ is infinite (since $\limgraphG$ has one end),
and thus at least one of the connected components of~$\limgraphG-S$ is finite.
Therefore, for sufficiently large~$j$ each minimum separator in the graph~$G_j$
has size at most $|S|$.
This implies that~$G_j$ has minimum degree at most~$|S|$ by Mader's Theorem (Theorem~\ref{lem:mad}) and thus~$\limgraphG$ has minimum degree~$|S|$. This also implies that~$G_j$ has minimum degree exactly~$|S|$ and thus the connectivity of $G_j$ and $\limgraphG$ coincides.

If $|S|=1$, then the minimum degree of $\limgraphG$ is 1, contradicting that $\limgraphG$ is edge-transitive, connected,
infinite and twin-free.
If $|S|=2$, then the minimum degree of $\limgraphG$ is 2, and therefore
$\limgraphG$ must have two orbits and the vertices in one of the orbits, say~$O$, have degree 2.
In this case, we can replace every path of length 2 that has an internal vertex from~$O$ by an edge and
obtain a graph~$\limgraphG'$ that is vertex-transitive and has the same automorphism group as~$\limgraphG$. 

Suppose now that~$|S|=3$.
If a finite connected component of~$\limgraphG-S$ were to contain more than one vertex,
then for sufficiently large~$j$
there is a separator in $G_j$ separating more than one vertex.
However, finite 3-connected edge-transitive, twin-free graphs are known to be almost-4-connected (see for example here~\cite[Theorem 1]{DBLP:journals/dam/ZhangM08}).
Therefore, the finite connected component of~$\limgraphG-S$ consists of only one vertex,
and thus the limit graph $\limgraphG$ is almost-4-connected.
\end{proof}

Recall that a graph is almost vertex-transitive
if it has only finitely many vertex orbits under its automorphism group. Note that edge-transitive graphs are almost vertex-transitive.
An end of a graph is \emph{thick} if it contains an infinite collection of
pairwise disjoint one-way infinite paths.  We will only use the concept in the following two theorems.
\begin{theo}[{\cite[Theorem~5.6]{DBLP:journals/combinatorica/Thomassen92}}] 
If~$G$ is a connected, infinite, locally finite and almost vertex-transitive graph
with only one end, then that end is thick.
\end{theo}

\begin{theo}[{\cite[Theorem~4.1]{DBLP:journals/combinatorica/Thomassen92}}]
Let~$G$ be a connected, infinite, locally finite, almost vertex-transitive, non-planar,
3-connected and almost-4-connected graph with at least one thick end. Then,~$G$ is
contractible into an infinite complete graph.
\end{theo}

Overall, we conclude that our limit graph~$\limgraphG$ is planar. We will use a theorem of Babai relating vertex-transitive planar graphs to Archimedean tilings.

\begin{theo}[{Babai~\cite[Theorem 3.1]{DBLP:conf/soda/Babai97}}]
Let~$G$ be a locally finite, connected, vertex-transitive planar graph with at most one end. Then,~$G$ has an embedding in a natural geometry as an Archimedean tiling. All automorphisms of~$G$ extend to automorphisms of the tiling and are induced by isometries of the geometry. 
\end{theo}

Here, the \emph{natural geometries} are the spheres, the Euclidean
plane and hyperbolic planes (with constant curvature). Their \emph{Archimedean tilings} are
tilings by regular polygons such that the group of isometries of the
tiling acts transitively on the vertices of the tiling.
The spherical geometries arise precisely when the graph is finite, which we can rule out since~$\limgraphG$ is infinite.

We need to deal with the fact that~$\limgraphG$ might not be vertex-transitive in our case, say having two vertex orbits~$O_1$ and~$O_2$. However, in that case, we can consider the graph~$\hat{G}$ obtained from~$\limgraphG$ by removing the vertices from~$O_2$ and joining two vertices~$v_1,v_2$ in~$O_1$ if they have a common neighbor in~$O_2$ and lie on a common face in the (up to reflection unique) planar embedding of a sufficiently large neighborhood of~$v_1$. Indeed, this follows from the infinite version of Whitney's theorem (see \cite{MR0384588} or \cite{MR0685062}) which says that 3-connected planar graphs have unique embedding and the fact that either~$\limgraphG$ is 3-connected or a subdivision of a 3-connected graph (Lemma~\ref{lem:1:end:connectivity}).
The graph~$\hat{G}$ is also planar by construction. It is vertex-transitive since~$\limgraphG$ is edge-transitive and the construction of~$\hat{G}$ is isomorphism invariant. (However,~$\hat{G}$ may have a larger Hadwiger number than~$\limgraphG$.) If~$\limgraphG$ only has one orbit, we simply define~$\hat{G}\coloneqq \limgraphG$.

We will use the following fact about hyperbolic spaces.
\begin{fact}\label{fact:exponential:growth:in:hyperbolic:space}
The circumference of balls in a hyperbolic plane (of constant curvature) grows exponentially with the radius of the ball.
In particular, for an Archimedean tiling of such a hyperbolic plane, the number of tiles at distance~$r$ from a point grows exponentially with the distance~$r$.
\end{fact}
We need some observations that, in the hyperbolic case, allow us to relate distances in the metric space to distances in the graph~$\limgraphG$.  The next lemma essentially says that distances measured in $\limgraphG$, $\hat{G}$, and in the natural geometry agree up to a constant factor and that there can be no dead ends of unbounded depth (see~\cite{DBLP:journals/ijac/Lehnert09} for more information on dead ends in groups).

\begin{lem}
Assume the Archimedean geometry of the tiling of the graph~$\hat{G}$ is hyperbolic. Consider the graphs $\limgraphG$, $\hat{G}$ and the hyperbolic metric space~$X$ into which~$\hat{G}$ has an embedding. Then, it holds that~$V(\hat{G})\subseteq X$,~$V(\hat{G})\subseteq V(\limgraphG)$ and there are positive constants~$c_1,c_2$ and~$c_3$ (depending on~$\hat{G}$) so that
\begin{enumerate}
\item\label{metric:1} for all vertices~$v,v'\in V(\hat{G})$   we have~$\frac{1}{c_1} \cdot d_{\hat{G}}(v,v')\leq d_{\limgraphG}(v,v')\leq c_1 \cdot d_{\hat{G}}(v,v')$,
\item\label{metric:2} for all vertices~$v,v'\in V(\hat{G})$   we have~$\frac{1}{c_2} \cdot d_{\limgraphG}(v,v')\leq d_{X}(v,v')\leq c_2 \cdot d_{\limgraphG}(v,v')$, and
\item\label{metric:3} for every pair of vertices~$v,v'\in V(\limgraphG)$
there is a vertex~$w\in V(\limgraphG)$ with~$d_{\limgraphG}(v',w)\leq c_3$ and~$d_{\limgraphG}(v,w)> d_{\limgraphG}(v,v')$.
\end{enumerate}
\end{lem}
\begin{proof}
Part \ref{metric:1} follows from the fact that~$\limgraphG$ is a 3-connected, edge-transitive, planar graph with finite maximum degree, and thus 
there is a uniform bound for the diameter of~$N(v)$ in~$\hat{G}$ across all~$v\in O_2$.

Part \ref{metric:2} follows from the fact that there is an absolute bound on the diameter of the tiles in the Archimedean tiling.

Finally, Part \ref{metric:3} follows from the fact that the statement
is true in the space~$X$ (since ever geodesic can be extended) and that every point in~$X$ is at bounded
distance from~$V(\limgraphG)$.
\end{proof}

We can now use Babai's sphere packing argument (see \cite{DBLP:journals/jgt/Babai91}) to rule out that the Archimedean geometry of~the vertex-transitive graph $\hat{G}$ is hyperbolic (i.e., has negative curvature).

\begin{lem}\label{lem:spherePacking}
If the  Archimedean geometry of the tiling of~$\hat{G}$ is hyperbolic, then~$\lim_{j\rightarrow \infty} \Had(G_j) =\infty$.
\end{lem}
\begin{proof}
By the previous lemma, distances in~$\limgraphG$ agree with distances
  in~$\hat{G}$ and with distances in the hyperbolic metric space~$X$
  up to a constant factor.
  We can therefore use~Fact~\ref{fact:exponential:growth:in:hyperbolic:space} to conclude the following. There are superlinear functions~$f_1,f_2\in  \omega(t)$ so that for~$t\in \mathbb{N}$ we can, 
  on the boundary of the ball of radius~$t$, that is in~$\partial B_{t,\limgraphG}(\overline v) \coloneqq B_{t,\limgraphG}(\overline v)\setminus B_{t-1,\limgraphG}(\overline v)= N(B_{t-1,\limgraphG}(\overline v))$, find~$f_1(t)$ distinct vertices with a pairwise distance of at least~$f_2(t)$ outside of~$B_{t-c_3,\limgraphG}(\overline v)$ (i.e., the distance is measured in~$\limgraphG-B_{t-c_3,\limgraphG}(\overline v)$). (Here we also use that in the hyperbolic plane~$X$, a path connecting points on the boundary of a ball which are shortest among all paths that do not enter the ball lies entirely in the boundary of the ball.)

  Let~$t\in \mathbb{N}$ be some integer.  Choose~$j_0$ sufficiently
  large so that for all~$j\geq j_0$ the balls of radius~$(3+c_3)t$ in~$G_j$ are
  isomorphic to balls of radius~$(3+c_3)t$ in~$\limgraphG$. Let~$S$ be an
  inclusion-wise maximal set in~$G_j$ of vertices of pairwise distance at
  least~$2t$.  Assign every vertex of~$G_j$ to a vertex in~$S$ of
  closest distance, ties broken arbitrarily.  
  This gives us a minor~$H$ of~$G_j$.  We claim that the minimum
  degree of~$H$ tends to infinity as~$t$ tends to infinity. This will
  show the statement by Theorem~\ref{theo:1}.

  Consider the ball~$B_{t,G_j}(v)$ around a vertex~$v\in S$ of
  radius~$t$. Choose a set~$Y$ of~$f_1(t)$ vertices in $G_j$ at distance~$t$ from~$v$ that have pairwise distance at least~$\min\{f_2(t),2(2+c_3)t+1\}$ outside of~$B_{t-c_3,G_j}(v)$. These exist since balls of radius~$(3+c_3)t$ in~$\limgraphG$ are isomorphic to balls of
  radius~$(3+c_3)t$ in~$G_j$ and vertices of~$Y$ are a distance of at least~$(2+c_3)t$ away from the border of the ball. (Thus, a shortest path from a vertex in~$Y$ leaving the ball of radius~$(3+c_3)t$ and coming back to a vertex in~$Y$ has length at least~$2 (2+c_3)t+1$.)
  
  For each~$y\in Y$ we define a vertex~$s_y$ in~$S\setminus {v}$ that is relatively close to~$y$ as follows.
  We start a walk in~$y_0= y$. We take  at most~$c_3$ steps to get to a vertex that is further away from~$v$ than~$y$. We repeat the process. 
  Overall we obtain a walk~$y_0,y_1,y_2,\ldots$
  with a subsequence~$y=y_{i_0},y_{i_1},y_{i_2}$ so that~$i_{j+1}-i_j\leq c_3$ and~$d(y_{i_{j+1}},v)> d(y_{i_{j}},v)$.
  Let~$y_i$ be the first vertex on this walk assigned to a vertex~$s_y\in S$ other than~$v$.
  This means that~$y_i$ belongs to the branch set of~$s_y$.
  Note that for the minor~$H$
  the branch set containing~$s_y$ is adjacent to the branch set of~$v$ since~$y_i$ is adjacent to~$y_{i-1}$.
  
  It suffices now to argue that the~$s_y$ are distinct since then the branch set~$v$ has~$f_1(t)$ neighbors.
  This can be seen as follows. The distance between~$s_y$ and~$y$ is at most~$tc_3+2t-1$. (Starting in~$y$ after at most~$c_3 t$ steps we reach a vertex of distance at least~$2t$ from~$v$ which cannot belong to the branch set of~$v$.
  From the first vertex~$y_i$ not in the branch set of~$v$, in at most~$2t-1$ steps we reach the vertex~$s_y\in S$.) This means that~$s_y$ and~$s_y'$ for distinct~$y,y'\in Y$ have distance at least~$\min\{f_2(t),2(2+c_3)t+1\}-2(tc_3+2t-1)$. For~$t$ sufficiently large, this number is positive, which shows that the~$s_y$ are distinct. 
  Overall, the number of neighbors a branch set in~$H$ has is at least~$f_1(t)=|Y|$ and grows as~$t$ grows. This implies that as~$t$ increases, the average degree of~$H$ and thus by Theorem~\ref{theo:1} the Hadwiger number of~$G_j$ increases as~$j$ increases.
\end{proof}

\begin{lem}\label{lem:char0}
  If the limit graph~$\limgraphG$ is planar with one end, then for sufficiently large~$j$, the graph~$G_j$ has Euler characteristic 0.
\end{lem}
\begin{proof}
  By \cref{lem:spherePacking}, we know that the graph~$\hat{G}$ can be
  interpreted as Archimedean tiling of the Euclidean plane.

Suppose~$\hat{G}$ has degree~$d$.
Pick an arbitrary vertex $v$ and let~$f_1,\ldots,f_d$ be the number of edges for each of the $t$ faces incident with~$v$.
Since $\hat{G}$ is vertex-transitive, the numbers $f_1,\ldots,f_d$ do not depend on $v$.
(We actually know that there can be at most two different face sizes, i.e., $|\{f_1,\ldots,f_d\}|\leq 2$, but we do not use this fact.)
The arguments in~\cite{DBLP:conf/soda/Babai97} in fact tell us that the curvature of the space on which the tiling acts can be described in terms of the face sizes around a vertex. In particular, we know that~$\sum_{i=1}^d(1/2-1/f_i)=1$, since otherwise the Archimedean tiling will not be on the Euclidean plane.

The arguments in~\cite[Section 6.2]{DBLP:journals/jgt/Babai91} show
that, for sufficiently large~$j$, the graph~$G_j$ has locally a unique
planar embedding. Babai further argues that these local embeddings are locally consistent and
overall give us an embedding of~$G_j$ into some surface. In analogy to
our previous operation, we can construct the graph~$\hat{G_j}$ by
removing vertices from orbit~$O_2$ and joining vertices of~$O_1$ if
they are at distance 2 and share a face.  As explained in~\cite[Section 6.2]{DBLP:journals/jgt/Babai91}, the
graph~$\hat{G_j}$ satisfies~$\sum_{i=1}^d(1/2-1/f_i)=1$. This
implies that the embedding of~$\hat{G_j}$ and thus the embedding
of~$G_j$ has Euler characteristic~0. (By double counting, a vertex of
degree~$d$ contributes 1 vertex,~$d/2$ edges and~$1/f_i$ faces
to each adjacent face, so in total~$\sum_{i\in[d]} 1/f_i$ to the
faces. Thus, the contribution to the Euler characteristic is $1-d/2+\sum_{i\in[d]} 1/f_i=0$ for each vertex).
\end{proof}

Overall, we have proven that for~$j$ sufficiently large the graph~$G_j$ admits an embedding on the torus or the Klein bottle.  Now, we can use Babai's classification~\cite{DBLP:journals/jgt/Babai91} or Thomassen's classification~\cite{MR1040045} for such graphs. 
In particular, the automorphism group~$\Aut(G_j)$ has an abelian normal subgroup 
of index at most 12 that is generated by at most two elements. Moreover, the automorphism group~$\Aut(G_j)$ is solvable.

\subsection{Infinitely many ends}\label{subsec:infinitely:many:ends}

For the vertex-transitive case with infinitely many ends, Babai proved the following theorem. 
\begin{theo}[{\cite[Theorem 5.4]{DBLP:journals/jgt/Babai91}}]
Suppose~$G_1,G_2,G_3,\ldots$ is a convergent sequence of finite
connected vertex-transitive graphs of bounded degree. If the limit graph~$\limgraphG$ of the sequence has more than two ends, then~$\lim_{j\to \infty}\Had(G_j) = \infty$.
\end{theo} 

An inspection of the proof of the theorem shows that it can be extended to the edge-transitive case.

\begin{theo}\label{thm:inf:ends:large:Hadwiger}
Suppose that~$G_1,G_2,G_3,\ldots$ is a convergent sequence of finite
connected edge-transitive graphs of bounded degree. If the limit~$\limgraphG$ of the sequence has more than two ends, then $\lim_{j\to \infty}\Had(G_j) = \infty$.
\end{theo} 

\begin{proof}
We can essentially apply Babai's proof \cite[Theorem 5.4]{DBLP:journals/jgt/Babai91} with obvious adaptations accounting for the possibility of two orbits as follows.
Choose $t$ such that for every vertex $\overline v\in V(\limgraphG)$ and ball~$B_{t,\limgraphG}(\overline v)$ the graph~$\limgraphG-B_{t,\limgraphG}(\overline v)$ has at least~$m$ infinite components (this is possible since there can only be two isomorphism types of balls).
Choose~$j_0$ so that for~$j\geq j_0$ balls of radius~$4t$ in~$G_j$ are isomorphic to balls of radius~$4t$ in~$\limgraphG$.
Let~$S$ be a maximal set of points in~$G_j$ of pairwise distance at least $2t$. 

Assign every vertex~$x\in V(G_j)$ to a vertex of~$s_x\in S$ of closest distance, ties broken arbitrarily. This gives us a minor~$H$ of~$G_j$, where two vertices are in the same branch set if they are assigned to the same vertex of~$S$.
We claim that the minimum degree of~$H$ is at least~$m$. 
This will show the statement by Theorem~\ref{theo:1}.

Pick a vertex $v\in S$.
In the limit graph~$\limgraphG$ the graph $\limgraphG-B_{t,\limgraphG}(\overline v)$ has at least~$m$ infinite components, so choose vertices $\overline v_1,\ldots,\overline v_m$ in pairwise different component of
$\limgraphG-B_{t,\limgraphG}(\overline v)$ each
at distance~$2t$ from~$\overline v\in V(\limgraphG)$ (the vertex corresponding to $v$).

Let~$v_1,\ldots,v_m\in V(G_j)$ be the vertices corresponding to $\overline v_1,\ldots,\overline v_m\in V(\limgraphG)$ in~$G_j$ obtained by the isomorphism of the balls from~$\limgraphG$ to~$G_j$. 
Note that the distance $\dist_{G_j}(v_i,s_{v_i})$ is at most $2t-1$, 
in particular, $v\neq s_{v_i}$.
Then, the distance between~$v$ and a vertex~$s_{v_i}$ is at least~$2t$ (since $v,s_{v_i}\in S,v\neq s_{v_i}$) and at most~$\dist_{G_j}(v,v_i)+\dist_{G_j}(v_i,s_{v_i})\leq 2t+(2t-1)< 4t$.
Therefore, in the limit graph $\limgraphG$ the vertex $s_{\overline v_i}$ (the vertex corresponding to $s_{v_i}$) lies within the same connected component of $\limgraphG-B_{t,\limgraphG}(\overline v)$ as the vertex $\overline v_i$, otherwise a shortest path between
$\overline v_i$ and $s_{\overline v_i}$ must cross a vertex $\overline b\in B_{\limgraphG}(\overline v)$, but then
$\dist_{\limgraphG}(\overline v_i,s_{\overline v_i})= \dist_{\limgraphG}(\overline v_i,\overline b)+\dist_{\limgraphG}(\overline b,s_{\overline v_i})\geq t + t$.
This implies that in the graph $G_j$ the~$s_{v_i}$ are all distinct because they lie in the ball of radius~$4t$.

For each~$i$ choose a shortest path from~$s_{v_i}$ to~$v$. Let~$x_i$ be the last vertex on that path for which~$s_{x_i}\neq v$. Such a vertex exists since~$s_{v_i}\neq v$. Note that the vertices~$s_{x_i}$ are all distinct
since they also lie in the same connected component as $s_{\overline v_i}$ in $\limgraphG-B_{t,\limgraphG}(\overline v)$. Moreover, the branch set corresponding to~$s_{x_i}$ is adjacent to the one corresponding to~$v$ since~$x_i$ is a vertex of the former adjacent to a vertex of the latter.
This shows that each branch set of~$H$ has at least~$m$ neighbors.
\end{proof}

\subsection{Two ends}\label{subsec:two:ends}
Suppose now that~$\limgraphG$ has two ends and let~$k$ be the
connectivity of~$\limgraphG$ between the two ends.

A map~$\varphi \colon V(G_1)\rightarrow V(G_2)$ from one graph~$G_1$ to another~$G_2$ is a \emph{local isomorphism} if for every vertex~$v\in G_1$ the restriction of~$\varphi$ to~$N_{G_1}(v)$ is a bijection between~$N_{G_1}(v)$ and~$N_{G_2}(\varphi(v))$.
A \emph{covering map}~$\cov \colon V(G_1)\rightarrow V(G_2)$ from one graph~$G_1$ to another~$G_2$ is a surjective local isomorphism. 
Note that generally if~$G_2$ is connected, then every local isomorphism to~$G_2$ is surjective.
A \emph{lift} of an automorphism~$\tau \in \Aut(G_2)$ 
 is an automorphism~$\tau^{\uparrow}\in\aut(G_1)$ so that for all~$v\in V(G_1)$ we have~$(v^{\cov})^{\tau} =(v^{\tau^{\uparrow}})^{\cov}$.
Note that if there is a covering map from $G_1$ to $G_2$ for which all automorphisms lift,
then there is a surjective homomorphism from the subgroup of $\aut(G_1)$ consisting of all lifts to the
automorphism group $\aut(G_2)$.

\begin{lem}\label{lem:two:ends:covering}
For sufficiently large~$j$ there is a covering map from~$\limgraphG$ to~$G_j$ for which all automorphisms lift. In particular, 
the group~$\aut(G_j)$ is a factor group of a subgroup of~$\aut(\limgraphG)$.
\end{lem}
\begin{proof}
  Note first that~$\lim_{j\to \infty} \Diam(G_j)=\infty$,
  where~$\Diam(G_j)$ denotes the diameter of~$G_j$.  Recall that~$B_{t,J}(v)$ denotes the ball around a vertex~$v$ of
  radius~$t$ in a graph~$J$. In the proof we will consider the balls~$B_{t,G_j}(v)$ and~$B_{t,\limgraphG}(v)$ for~$t\geq t_0$ and~$j\geq j_0$ sufficiently large. The requirements for~$t_0$ are given in the proof and depend only on~$\limgraphG$. The value~$j_0$ depends on~$t_0$ and the sequence of graphs~$G_j$ and ensures that the balls are isomorphic.
  
  Observe that in~$\limgraphG$, because it has two ends, the vertices on the boundary of a sufficiently large ball can be partitioned into two sides.
  More precisely, we claim the following.
  
  \begin{claim}\label{clam:border:distances}
  
  There is a constant $d_0$ (depending only
        on~$\limgraphG$) such that for all $d_1$  there is a $t_0$ (depending only
      on~$\limgraphG$ and $d_1$) such that 
    for~$t\geq t_0$
    the ball~$B_{t,\limgraphG}(v)$ has the following properties. The vertices on the boundary~$\partial B_{t,\limgraphG}(v) \coloneqq B_{t,\limgraphG}(v)\setminus B_{t-1,\limgraphG}(v) = N(B_{t-1,\limgraphG}(v))$
    can be partitioned
    into two sets~$Y_1,Y_2$ so that for~$i\in \{1,2\}$ we have $\forall u,u'\in Y_i : d_{\overline{G}}(u,u')\leq d_0$,
    and~$\forall u\in Y_1,u'\in Y_2 :d_{\overline{G}}(u,u')>
    d_1$.
    \end{claim}

 \begin{claimproof} This follows from the structure of~$\overline{G}$ (Lemma~\ref{lem:levelSets}) as follows.
 Let $v_0$ be a vertex that is adjacent (or equal) to $v$ and that
 is contained in a (primary) level set. Without loss of generality $v_0\in L_0$.
 Let~$D$ be the maximum distance of vertices within the same (primary or secondary) level set measured in~$\overline{G}$.
 Let~$b\in\{\frac{1}{2},1\}$ be the constant such that~$1/b$ is the number of orbits of~$\limgraphG$ (and $1/b$ is also the distance between two level sets $L_i$ and $L_{i+1}$).
 Then, the level sets $L_{-b(t+1)}$ and $L_{b(t+1)}$ are disjoint from the ball $B_{t,\limgraphG}(v_0)$,
 while the level sets $L_{-bt}$ and $L_{bt}$ intersect the boundary of the ball for all $t\in\NN$ (where $L_{i+\frac{1}{2}}\coloneqq J_{i+\frac{1}{2}},i\in\ZZ$).
 Therefore, if $t_0$ is sufficiently large ($t_0\geq D$), then for all $t\geq t_0$ it holds that $L_{-b(t-D)}$ and $L_{b(t-D)}$ is entirely contained in~$B_{t,\limgraphG}(v_0)$.
  We conclude that $[L_{-b(t-D)},L_{b(t-D)}]\subseteq
  B_{t,\limgraphG}(v_0)\subseteq[L_{-bt},L_{bt}]$ for $t\geq t_0$ and sufficiently large $t_0$.
  Therefore, the boundary
  of the ball can be partitioned into two subsets $Y_1,Y_2$ such that
  $Y_1\subseteq [L_{-bt},L_{-b(t-D)}]$ and
  $Y_2\subseteq [L_{b(t-D)},L_{bt}]$.  Then, for all $t\geq D$ it holds that
  $d_{\limgraphG}(u,u')\leq \frac{1}{b}bD+D$ for $u,u'\in Y_1$ (or
  $u,u'\in Y_2$) and $d_{\limgraphG}(u,u')> \frac{1}{b}2b(t-D)$ for $u\in Y_1$ and
  $u'\in Y_2$. This means that the distance between $Y_1$ and $Y_2$ grows at least linearly in $t$,
  while distances in $Y_1$ and in $Y_2$ are bounded by a constant (not depending on $t$).
  Since the vertices $v$ and $v_0$ (and their boundaries) have distance at most 1,
  the claim holds for sufficiently large $t_0$.
\end{claimproof}
 Since we can assume~$\Diam(G_j)$ is large,
for every~$t\geq t_0$ the claim
also
  holds for all~$G_j,j\geq j_0$ in place of~$\limgraphG$ whenever $j_0$ is sufficiently large. We call these two
  equivalence classes the borders of the ball~$B_{t,G_j}(v)$. Let us call
  one of these borders the \emph{left border} and the other the \emph{right
    border}.

  Let~$D$ be the maximum distance between vertices in a minimum
  separator of~$\limgraphG$ separating the ends. Recall that~$D$ is
  finite by~\cref{lem:boundedDiameter}.
  Let~$A\coloneqq A_{t,G_j}(v)\subseteq B_{t,G_j}(v)$ be the set of vertices in the ball that have distance
  more than~$d+D$ from vertices in the boundary~$\partial B_{t,G_j}(v)$, where~$d$ is the constant appearing in Claim~\ref{clam:border:distances}. This implies that minimum
  separators which contain a vertex from~$A$ and which separate the borders of~$B_{t,G_j}(v)$ are completely contained in~$B_{t-1,G_j}(v)$. Moreover, such separators do not separate any vertices in the same border.
  In particular, they have cardinality~$k$, where~$k$ is the connectivity between the two ends
  of~$\limgraphG$.
  Define~$V^{G_j,t}_{\LeftMostSep}$ to be the set of vertices of~$A$
  contained in a minimum separator separating the two borders
  of~$B_{t,G_j}(v)$.

  (In the following, we intuitively construct a rotation of~$G_j$ that
  does not rotate by too much, but it is cumbersome to define what a
  rotation is.) Choose~$t_0$ sufficiently large (and increase~$j_0$
  adequately) so that for all $t\geq t_0$ the set~$V^{G_j,t}_{\LeftMostSep}$ contains three
  vertices~$v_1,v_2,v_3$ all in the same orbit and pairwise at distance larger than~$2D$.
  For~$i\in \{2,3\}$, let~$\psi_{1,i}$ be an automorphism of~$G_j$
  mapping~$v_1$ to~$v_i$.  Let~$S$ be a separator of~$B_{t,G_j}(v)$ containing a vertex~$x\in A$ and separating the left and right border of the ball. Each such minimum separator~$S$ is adjacent to
  a connected component of~$G_j[B_{t,G_j}(v)]-S$ containing the left border and another connected component containing the right border.
  We call these the left and right components of the
  separator, respectively.  Let~$S_1$ be such a minimum separator containing~$v_1$.
   Let~$e$ be an edge that has one endpoint
   in~$S_1$ and whose other endpoint is in the left component of~$S_1$.
  Define~$\psi$ as follows: If~$e^{\psi_{1,2}}$ is in the left
  component of~$S_1^{\psi_{1,2}}$ (meaning~$\psi_{1,2}$ behaves like a
  rotation) then~$\psi\coloneqq \psi_{1,2}$. Otherwise,
  if~$e^{\psi_{1,3}}$ is in the left component of~$S_1^{\psi_{1,3}}$
  (meaning~$\psi_{1,3}$ behaves like a rotation)
  then~$\psi\coloneqq \psi_{1,3}$. Otherwise,
  set~$\psi \coloneqq   \psi_{1,3}^{-1}\psi_{1,2}$ (functions
  applied from left to right).
  Note that overall~$\psi$ maps edges reaching into the left component
  to edges reaching into the left component (i.e., it behaves like a
  rotation).  It also maps some vertex~$a\in \{v_1,v_2,v_3\}$ to some
  distinct vertex~$a^{\psi}\in \{v_1,v_2,v_3\}$ at distance larger
  than~$2D$. In particular, minimum separators containing~$a$ are
  disjoint from minimum separators containing~$a^\psi$.

  Choose a vertex~$v^{\uparrow}$ in~$\limgraphG$ so that~$\limgraphG[B_{t,\limgraphG}(v^{\uparrow})]$
  and~$G_j[B_{t,G_j}(v)]$ are isomorphic as rooted graphs, and
  let~$\varphi$ be an
  isomorphism from~$B_{t,\limgraphG}(v^{\uparrow})$ to~$B_{t,G_j}(v)$
  mapping~$v^{\uparrow}$ to~$v$. 
  Recall that~$\limgraphG -B_{t,\limgraphG}(v^{\uparrow})$ has two infinite connected components. To ensure our notions of left and right are consistent in~$\limgraphG$ and~$G_j$, we define the left component in~$\limgraphG -B_{t,\limgraphG}(v^{\uparrow})$ as the component whose neighbors are mapped to the left boundary of~$B_{t,G_j}(v)$ by~$\varphi$ and similarly for the right. 
  For the vertex~$a$ of the previous paragraph, define~$a^{\uparrow}$
  and~$(a^\psi)^{\uparrow}$ so that~$(a^{\uparrow})^\varphi=a$
  and~$((a^\psi)^{\uparrow})^\varphi=a^\psi$. In~$\limgraphG$ find an
  automorphism~$\psi^{\uparrow}$ of~$\limgraphG$ which maps~$a^{\uparrow}$
  to~$(a^\psi)^{\uparrow}$ does not interchange the
    ends.
  If such an automorphism does not exist, we use five (distinct) points
    $a^{\uparrow},(a^\psi)^{\uparrow},\ldots,(a^{\psi^4})^{\uparrow}$. Three of these must be in the same orbit, say~$a_1^{\uparrow},a_2^{\uparrow},a_3^{\uparrow}$. If still no automorphism between the points exists that does not interchange the ends, we again assemble two reflections to a rotation (finding isomorphisms~$\psi_{i,j}^{\uparrow}$ mapping~$a_i^{\uparrow}$ to~$a_j^{\uparrow}$ and considering~$(\psi_{1,3}^{\uparrow})^{-1}\psi_{1,2}^{\uparrow}$).

  Possibly renaming various points and replacing isomorphism by up to their fourth power, we can now
  assume that
  \begin{itemize}
  \item $a^{\uparrow}$ is mapped to~$(a^\psi)^{\uparrow}$ by an isomorphism~$\psi^{\uparrow}$ that a does not interchange the ends and
  \item $\varphi$ maps $a^{\uparrow}$ to~$a$ and~$(a^\psi)^{\uparrow}$ to~$(a^\psi)$.
  \end{itemize}

Recall that for each~$v\in V(\limgraphG)$ contained in a minimum separator separating the two ends we denote by~$S_{v}$ the leftmost minimum separator containing~$v$. We give an analogous definition for~$G_j$.
  For each vertex~$a'$ of~$G_j$
  contained in~$A$ and contained in some minimum separator separating the two borders of~$B_{t,G_j}(v)$,
  we can define a
  \emph{leftmost minimum separator~$S_{a'}$} containing~$a'$ as the separator
  separating the borders of~$B_{t,G_j}(v)$
  for which the left component has
  minimum order.
  Since leftmost separators are unique in~$\limgraphG$, the
  separator~$S_a$ is also unique in~$G_j[B_{t,G_j}(v)]$.

  \[
    \begin{tikzcd}
      S_{a^{\uparrow}} \arrow{r}{\psi^{\uparrow}} \arrow[swap]{d}{\phi} & S_{(a^{\uparrow})^{\psi^{\uparrow}}} \arrow{d}{\phi}  \\%
      S_{a} \arrow{r}{\psi}& S_{a^\psi}
    \end{tikzcd}
  \]

  Note that~$(S_{a^{\uparrow}})^{\phi}= S_a$ and
  that~$(S_{(a^\psi)^{\uparrow}})^\phi= S_{a^\psi}$. Also note
  that~$(S_{a^{\uparrow}})^{\psi^{\uparrow}} = S_{(a^\psi)^{\uparrow}}= S_{(a^{\uparrow})^{\psi^{\uparrow}}}$
  and~$(S_{a})^\psi = S_{{a^\psi}}$.
  This means that when
  mapping~$S_{a^{\uparrow}}$ as a set, the functions~$\phi$ and
  $\psi$ (respectively~$\psi^{\uparrow}$, see the diagram)
  commute, i.e.,~$(S_{a^{\uparrow}})^{\phi\psi}= (S_{a^{\uparrow}})^{\psi^{\uparrow}\phi}$. However, we want them to commute when applied to the points
  in~$S_{a^{\uparrow}}$ separately.

  \begin{claim}\label{claim:shifting:a:and:z}
  If~$t_0$ is sufficiently large
  (and~$j_0$ increased adequately), then there are
  distinct~${\ell},{\ell}'\in\mathbb{N}$
  such that~$a^{\psi^{\ell}},a^{\psi^{{\ell}'}}\in A$
  and such that for all~$z\in (S_{a^{\uparrow}})^{(\psi^{\uparrow})^{\ell}}$ it holds
  that~$z^{\phi\psi^{{\ell}'-{\ell}}}=
   z^{{(\psi^{\uparrow})}^{{\ell}'-{\ell}}\phi}\in(S_{a^{\uparrow}})^{(\psi^{\uparrow})^{\ell'}}$.
  \end{claim}
  
  \begin{claimproof}
  Choose~$t_0$ sufficiently large
  so that for all~$0\leq s\leq k!$ we
  have~$(a^{\uparrow})^{\psi^{s}}\in A$, where~$k$ is the
  connectivity between the two ends of~$\limgraphG$, and thus the size of~$S_{a^{\uparrow}}$. For all~$s$ in
  this range the
  map~$\varphi \psi^{s} \varphi^{-1} {\psi^{\uparrow}}^{-s} $ defines a
  permutation of~$S_{a^{\uparrow}}$. For some pair~${\ell}\neq{\ell}'$
  these permutations agree, and these parameters show the claim.
  \end{claimproof}

  Let us replace $a$ by $a^{\psi^{\ell}}$
  and~$\psi$ by $\psi^{{\ell'}-{\ell}}$
  (and thus $a^\psi$ becomes $a^{\psi^{\ell'}}$). So, we are back to the original notation with the
  additional property that when applied to the points
  of~$S_{a^{\uparrow}}$ the functions~$\phi$ and $\psi$
  (respectively~$\psi^{\uparrow}$) commute.

  \emph{(Fundamental domain)} Let~$I= [S_{a^{\uparrow}}, S_{(a^{\uparrow})^{\psi^{\uparrow}}}]$
  be the interval between the two separators~$S_{a^{\uparrow}}$
  and~$S_{(a^{\uparrow})^{\psi^{\uparrow}}}$, i.e., the set of
  all vertices that are not in an infinite component
  of~$\limgraphG-(S_{a^{\uparrow}}\cup
  S_{(a^{\uparrow})^{\psi^{\uparrow}}})$.
Recall that~$\psi^{\uparrow}$ does not interchange the ends and that~$a^{\uparrow}$ and ~$(a^{\uparrow})^{\psi^{\uparrow}}$ have distance larger
than~$2D$. This implies that~$\psi^{\uparrow}$ does not fix the level sets of~$\limgraphG$. Thus, every vertex~$x$ in one of the two connected components of~$\limgraphG - I$ is mapped to some vertex in the other connected component by some (possibly negative) power of~${\psi^{\uparrow}}$. But~${\psi^{\uparrow}}$ (without taking powers) does not map vertices from one of the components to the other, so for some power of~${\psi^{\uparrow}}$ the vertex~$x$ is mapped to~$I$.  
Therefore,~$\bigcup_{i\in \mathbb{Z}} 
I^{{(\psi^{\uparrow})}^i} = V(\limgraphG)$. (The set~$I\setminus S_{(a^{\uparrow})^{\psi^{\uparrow}}}$ actually contains exactly one vertex of each orbit of~${\psi^{\uparrow}}$ making it a fundamental domain, but we will not need this fact.)

  \emph{(The covering map)} Define a
  function~$\cov\colon V(\limgraphG)\rightarrow V(G_j)$ as follows:
  for each~$w\in V(\limgraphG)$ find an~$s\in\mathbb{Z}$ such
  that~$w^{{\psi^{\uparrow}}^s}\in I$. Then,
  set~$w^{\cov} \coloneqq w^{{\psi^{\uparrow}}^{s} \phi\psi^{-s}}$.
  For vertices
  that can be mapped to~$S_{a^{\uparrow}}\cup S_{(a^{\uparrow})^{\psi^{\uparrow}}}$ the
  integer~$s$ might not be unique, but 
  Claim~\ref{claim:shifting:a:and:z} precisely says that this
  definition is well-defined since the functions~$\phi$ and $\psi$
  commute on vertices from~$S_{a^{\uparrow}}$.
  Since~$\cov$ is injective on~$I$ and both~$S_{a}$
  and~$S_{a^\psi}$ are separators, it follows that~$\cov$ is a
  covering map, i.e., a surjection that is a local isomorphism. For
  the surjectivity we use that~$G_j$ is connected. (The
  functions~$\cov$ and~$\varphi$ agree on $I$, but they may disagree outside of~$I$.)
  Let~$B^{\uparrow}\subseteq V(\limgraphG)$ with~$I\subseteq
  B^{\uparrow}$ be a set of vertices so that the restriction
  map~$\cov|_{{B^{\uparrow}}} \colon B^{\uparrow}\rightarrow V(G_j)$
  induces an isomorphism from~$\limgraphG[B^{\uparrow}]$
  to~$G_j[B_{t,G_j}(v)]$ (i.e.,~$\cov(B^{\uparrow}) =
  B_{t,G_j}(v)$).

  \emph{(All automorphisms lift)} We argue that all
  automorphisms of~$G_j$ lift to automorphisms of~$\limgraphG$. By
  construction~$w^{{\psi^{\uparrow}}\cov}= w^{\cov\psi}$ for all $w\in V(\limgraphG)$, and thus
  the automorphism~$\psi$ lifts to~$\psi^{\uparrow}$. 
  Since~$\bigcup_{i\in \mathbb{Z}} I^{{(\psi^{\uparrow})}^i} =
  V(\overline{G})$ implies~$\bigcup_{i\in \mathbb{Z}} I^{{\psi}^i} =
    V(G_j)$, it suffices to lift automorphisms~$\tau$ that map some vertex of~$I^{\varphi}$ to~$I^{\varphi}$.
  Since~$I^{\varphi}$ has bounded diameter,
  by possibly increasing~$t_0$ we can ensure that for every such map we have~$I^{\varphi \tau}\subseteq B_{t,G_j}(v)$.
   
  Let~$\tau$ be such an automorphism. Consider first the case that
  some edge that is incident with~$a$ and with an endpoint in the left
  component of~$S_a$ is mapped to an edge with an endpoint in the
  left component of~$(S_a)^\tau$ (i.e.,~$\tau$ behaves like a
  rotation).

  In the following we want to argue that we can bound the distance between~$I^{\varphi{\psi}^{i}\tau}$ and~$I^{\varphi{\psi}^i}$ independent of~$i$ and~$\tau$. We do this with the following claim.
  
  \begin{claim}\label{clam:interval:distance:constant}
    For some constant~$b$ independent of~$\tau$ for all~$i\in \mathbb{Z}$ we have~$I^{\varphi{\psi}^{i}\tau}\subseteq \bigcup_{\ell\in \{i-b,\ldots,i+b\}} I^{\varphi{\psi}^\ell}$. 
      \end{claim}
  
  \begin{claimproof}

  For a set of vertices~$M$ contained in a
  ball~$B_{t,G_j}(m)$ with~$m\in M$ we say a vertex~$u\in
  B_{t,G_j}(m)$ is \emph{surrounded}  by~$M$ if~$u \notin M$ and~$u$ is not in a connected component
  of~$B_{t,G_j}(m)\setminus M$ that contains a vertex in a
  border of~$B_{t,G_j}(m)$. Note that if~$t$ is sufficiently large
  (and~$j$ sufficiently large in dependence of that) then whether a
  vertex is surrounded by~$M$ is independent of the choice of
  the center~$m\in M$ of the ball. In fact, a lower bound
  which ensures that~$t$ is sufficiently large can be given in terms
  of the diameter of~$M$ measured in~$G_j$.

    Let~$b'$ be the maximum number of vertices surrounded by~$M_1\cup M_2$ for two non-trivially intersecting sets~$M_1$ and~$M_2$ each with a diameter at most that of~$I$ (where the diameter is measured in~$G_j$). We set~$b=2 b'+|I|+2$.

  We argue that the number of vertices surrounded by~$I^{\varphi{\psi}^{i}\tau}\cup I^{\varphi{\psi}^{i+b}}$ is a constant~$c> b'$ independent of~$i$. 
  We argue this statement by induction on~$|i|$.
   \begin{itemize}
 \item For the induction base with~$i=0$, recall that~$I^{\varphi\tau}$ and~$I^{\varphi}$ intersect nontrivially. There are at least~$b-1$ vertices surrounded by~$I^{\varphi}\cup I^{\varphi{\psi}^{b}}$.
 If a vertex surrounded by~$I^{\varphi}\cup I^{\varphi{\psi}^{b}}$ is not surrounded by~$I^{\varphi\tau} \cup I^{\varphi{\psi}^{b}}$ then it must be in~$I^{\varphi\tau}$ or it must be surrounded by~$I^{\varphi\tau}\cup I^{\varphi}$.
 This means at least~$b-1-|I|-b'> b'$ vertices
 are surrounded by~$I^{\varphi\tau}\cup I^{\varphi{\psi}^{b}}$.

  \item For the induction step assume first~$i>0$ and
  suppose~$c$ vertices are surrounded by $I^{\varphi{\psi}^{i}\tau}\cup  I^{\varphi{\psi}^{i+b}}$.
  Let~$e$ be the number of vertices surrounded by $I^{\varphi{\psi}^{i}}\cup   I^{\varphi{\psi}^{i+3}}$ minus the number of vertices surrounded  by $I^{\varphi{\psi}^{i}}\cup I^{\varphi{\psi}^{i+2}}$. This number is independent of~$i$. (Here we consider~$I^{\varphi{\psi}^{i+2}}$ instead of~$I^{\varphi{\psi}^{i+1}}$ because~$I^{\varphi{\psi}^{i}}$ and~$I^{\varphi{\psi}^{i+2}}$  are disjoint.)
  
   Then there are~$c+e $ vertices surrounded by $I^{\varphi{\psi}^{i}\tau}\cup   I^{\varphi{\psi}^{i+b+1}}$.
  
  Since~$\tau$ is an automorphism, $e$ is also exactly the number of vertices surrounded by~$I^{\varphi{\psi}^{i}\tau}\cup I^{\varphi{\psi}^{i+3}\tau}$ minus the number of vertices surrounded by~$I^{\varphi{\psi}^{i+1}\tau}\cup I^{\varphi{\psi}^{i+3}\tau}$.
  
  It follows that exactly~$c+e-e=c$ vertices are surrounded  by~$I^{\varphi{\psi}^{i+1}\tau}\cup I^{\varphi{\psi}^{i+b+1}}$.

  The inductive  argument for~$i<0$ is similar.
  \end{itemize}
  
  We conclude that the number of vertices surrounded by~$I^{\varphi{\psi}^{i}\tau}\cup I^{\varphi{\psi}^{i+b}}$ is a constant~$c> b'$ independent of~$i$. 
  Note that $c> b'$ implies that~$I^{\varphi{\psi}^{i}\tau}$ and
  $I^{\varphi{\psi}^{i+b}}$ do not intersect.
 
  Symmetrically, we can argue that the number of vertices surrounded by~$I^{\varphi{\psi}^{i}\tau}\cup I^{\varphi{\psi}^{i-b}}$ is constant.

  Since the number of vertices surrounded by~$I^{\varphi{\psi}^{i+1}\tau} \cup I^{\varphi{\psi}^{i+b+1}}$ and the number of vertices surrounded by~$I^{\varphi{\psi}^{i}\tau} \cup I^{\varphi{\psi}^{i-b}}$ is constant, we conclude, again by induction on~$|i|$, that~$I^{\varphi{\psi}^{i}\tau}\subseteq I^{\varphi{\psi}^{i-b}}\cup I^{\varphi{\psi}^{i-b+1}}\cup\cdots\cup I^{\varphi{\psi}^{i+b}}$.
  \end{claimproof}

  We define~$\tau^{\uparrow}$ as follows. For~$w\in V(\limgraphG)$ we
  find an~$s\in \mathbb{Z}$ so that~$w^{{(\psi^{\uparrow})}^s}\in
  I$. Then~$w^{{\psi^{\uparrow}}^s \cov {\psi}^{-s} \tau {\psi}^{s}} \in \bigcup_{\ell\in \{-b,\ldots,b\}} I^{\varphi{\psi}^{\ell}}\subseteq B_{t,G_j}(v)$ if we choose~$t_0$ sufficiently large.

  Define~$w^{\tau^{\uparrow}} \coloneqq
     w^{{\psi^{\uparrow}}^s \cov {\psi}^{-s} \tau {\psi}^{s} (\cov|_{{B^{\uparrow}}})^{-1}
       {\psi^{\uparrow}}^{-s}}$.
     Then,~$\tau^{\uparrow}$ is a
  lift of~$\tau$. Indeed, it is well-defined, despite possible choices
  for~$s$, again due to Claim~\ref{claim:shifting:a:and:z}.  It is an automorphism since all
  involved maps are isomorphisms when restricted to balls of suitable radius and~$\overline{G}$ is a connected strip.
  (Endomorphisms in a connected strip that are isomorphisms when restricted to sufficiently large
  balls are automorphisms.)

  Next, let~$\tau$ be an automorphism such that for some edge (and thus
  every edge) incident with~$a$ that has an endpoint in the left
  component of~$S_a$ its image has an endpoint in the right
  component of~$(S_a)^\tau$ (i.e.,~$\tau$ behaves like a
  reflection). Note that for~$t_0$ sufficiently large and $t\geq t_0$, the borders of the
  two balls~$B_{t,G_j}(v)$ and~$B_{t,G_j}(v')$ for adjacent
  vertices~$v$
  and~$v'$ are so that for each border its vertices are close (of
  distance at most~1) to exactly one border of the other ball. We can
  thus consistently label the borders with left and right for all
  balls of the graph so that left borders of adjacent vertices are
  adjacent (and similar for right borders).

  Then, for two balls of adjacent vertices the automorphism must
  interchange the borders of both of them or of neither.  Thus, the
  automorphism~$\tau$ interchanges the borders of all balls.

We define~$\tau^{\uparrow}$ as follows. For~$w\in V(\limgraphG)$
we find an~$s\in \mathbb{Z}$ so that~$w^{(\psi^{\uparrow})^s}\in I$. Then~$ w^{{\psi^{\uparrow}}^s \cov {\psi}^{-s} \tau {\psi}^{-s}}\in  \bigcup_{\ell\in \{-b,\ldots,b\}} I^{{\varphi(\psi^{\uparrow})}^{\ell}}\subseteq B_{t,G_j}(v)$ with arguments similar to the previous case.
 We also define~$w^{\tau^{\uparrow}} =   w^{{\psi^{\uparrow}}^s\cov{\psi}^{-s}\tau{\psi}^{-s}(\cov|_{{B^{\uparrow}}})^{-1}{\psi^{\uparrow}}^{s}}$. Then, with same arguments as before~$\tau^{\uparrow}$ is well-defined, a homomorphism, and locally an isomorphism. It is thus an automorphism, and in particular it is a lift of~$\tau$.

Overall this means that~$\limgraphG$ is a covering of~$G_j$ with
covering map~$\cov$ and all automorphisms lift. 
\end{proof}

\subsection{Combination of the results}\label{subsec:comb:results}

We finally assemble our considerations for varying number of ends to prove Theorem~\ref{thm:edge:transitive:twin:free:comp:fac:bound}.
\begin{proof}[Proof of Theorem~\ref{thm:edge:transitive:twin:free:comp:fac:bound}]
Assume for the sake of contradiction that there is some $h$ so that there is no~$f(h)$ such that 
the automorphism group of every connected, $K_{h+1}$-minor-free, edge-transitive, twin-free, finite graph is in~$\Gamma_{f(h)}$.
As argued before, there is an infinite convergent subsequence~$G_1,G_2,G_3,\dots$
for which there is no such~$f(h)$,
and this subsequence has a corresponding infinite limit graph $\limgraphG$. 
If $\limgraphG$ has one end, then Lemma~\ref{lem:char0} and the discussion thereafter says that~$\Aut(G_j)$ is solvable for sufficiently large~$j$, and thus in~$\Gamma_{f(h)}$ for~$f(h)=1$.
If it has infinitely many ends, then Theorem~\ref{thm:inf:ends:large:Hadwiger} says that the Hadwiger number of the graphs in the subsequence is not bounded. 
If it has two ends, then Lemma~\ref{lem:two:ends:covering} and Lemma~\ref{lem:comp:factors:of:two:end:vertex:trans} show that we can choose~$f(h)= h+1$ for sufficiently large~$j$.
\end{proof}

\section{Graphs with multiple edge orbits}\label{sec:multiple:edge:orbits}

We now turn to graphs that are not edge-transitive.

Recall that a minor $H$ is called invariant
if there is a minor model $\varphi\colon V(H)\rightarrow 2^{V(G)}$
so that~$V(H)^\phi$ and $E(H)^\phi=\{v^\phi w^\phi\mid vw\in E(H)\}$
are both invariant under automorphisms of $G$.
A vertex-colored minor is a pair $H_{\chi'}=(H,\chi')$ consisting of a vertex and a vertex coloring of its vertices.
A vertex-colored minor $H_{\chi'}=(H,{\chi'})$ is invariant
if additionally $({\chi'})^\phi$ is \emph{invariant} under automorphisms of $G$, that is, branch sets of a particular color under~$\chi$ must be mapped to branch sets of the same color. If~$G$ is also vertex-colored the requirement only needs to hold for automorphisms preserving the colors of~$G$.

\begin{lem}\label{lem:aut:invar:minor}
If~$H_{\chi'}=(H,\chi')$ is an $\aut(G_\chi)$-invariant vertex-colored minor of some vertex-colored graph~$G_\chi=(G,\chi)$, then there is a homomorphism from $\aut(G_\chi)$ to~$\aut(H_{\chi'})$ whose kernel is a subgroup of the direct product of 
automorphism groups of the vertex-colored graphs induced by the branch sets.
\end{lem}

\begin{proof}
Let $\phi\colon V(H)\to 2^{V(G)}$ be an $\aut(G_\chi)$-invariant minor model of $H$ in $G$.
Since $V(H)^\phi$, $E(H)^\phi$ and $(\chi')^\phi$ are invariant under $\aut(G_\chi)$, the group $\aut(G_\chi)$ acts on $V(H)$ via $v^\gamma\coloneqq v^{\phi\gamma\phi^{-1}}$ for $v\in V(H)$,
preserving edges and colors.
This leads to a homomorphism $g\colon\aut(G_\chi)\to\aut(H_{\chi'})$
where the kernel of $g$ is a subgroup of the direct product $\bigtimes_{v\in V(H)}\aut(G_\chi[\phi(v)])$.
\end{proof}

\begin{lem}\label{lem:smallest:orbit:is:gamma:d}
There is a function~$f$ such that if~$G_\chi=(G,\chi)$ is a connected, $K_{h+1}$-minor-free
vertex-colored graph,
then for every vertex orbit~$O$ of minimum cardinality the subgroup induced by~$\aut(G_\chi)$ on~$O$ is a~$\Gamma_{f(h)}$-group.
\end{lem}
\begin{proof}
\icase{1. $G_\chi$ is edge-transitive}

To apply Theorem~\ref{thm:edge:transitive:twin:free:comp:fac:bound}, we need to get rid of twins.
We call the set of twins of a vertex the \emph{twin class} of this vertex.
Let $O_1$ be a minimum cardinality vertex orbit.

\begin{claim}
The size $c$ of the twin class of a vertex in $O_1$ is at most $h$.
\end{claim}
\begin{claimproof}
If $G_\chi$
is vertex-transitive and has twin classes of size $c$,
then $c=|V(G)|\leq h$ or $G$ has the complete bipartite graph $K_{c,c}$ as a minor.
In this case, the complete graph $K_c$ is a minor (a matching is a minor model of $K_c$ in
$K_{c,c}$), and thus $c\leq h$. Assume that $G_\chi$ has exactly two orbits $O_1,O_2$ where
$|O_1|\leq|O_2|$. Let $G'$ be the graph obtained from $G$ by
removing all but one twin in each twin class in $O_1$ and let
$O_1',O_2$ be its orbits.
Note that $|O_1'|=\frac{1}{c}|O_1|\leq\frac{1}{c}|O_2|$.
Since $G'$ is biregular and every vertex in $O_2$ has a neighbor in $O_1$, the degree of each vertex in $O_1'$ is at least
$c$.
Therefore, the graph $G$ has a $K_{c,c}$-minor implying that $c\leq h$.
\end{claimproof}

Let $G''$ be the (uncolored) graph that is obtained from $G$ by removing all but one twin in each twin class (of $O_1$ as well as $O_2$).
Note that $G''$ is connected, $K_{h+1}$-minor-free, twin-free, edge-transitive and uncolored.
This allows us to
apply \cref{thm:edge:transitive:twin:free:comp:fac:bound} to $G''$. Thus, it holds that~$\Aut(G'')\in \Gamma_{f(h)}$.
Since the class $\Gamma_{f(h)}$ is closed under subgroups
and wreath products with symmetric (base) groups $S_c\in\Gamma_{f(h)}$ where~$c\leq h\leq f(h)$, the result also follows for the vertex-colored graph $G_\chi$
with twin classes in $O_1$ of size at most $c$. (Note that~$h\leq f(h)$ due to the graph~$K_h$.)

\icase{2. $G_\chi$ is not edge-transitive but vertex-transitive}

Let $\hat E\subseteq E(G)$ be an edge $\aut(G_\chi)$-orbit and consider the graph $\hat G\coloneqq(V(G),\hat E)$ induced by~$\hat E$.
Let $G^*_\chi\coloneqq(G^*,\chi^*)$ be the vertex-colored minor of $G_\chi$ obtained by contracting the edges~$\hat E$
where~$\chi^*$ is the vertex-coloring induced by $\chi$ (the connected components of $\hat G$
get assigned a fresh color,
and the remaining vertices keep their old color according to $\chi$).
Note that $G^*_\chi$ is an $\aut(G_\chi)$-invariant vertex-colored minor of $G_\chi$.
By induction on the order of~$G_\chi$, it holds that $\aut(G^*_\chi)\in \Gamma_{f(h)}$
and $\aut(G_\chi[Z])\in\Gamma_{f(h)}$ for each connected component $Z$ of $\hat G$.
By Lemma~\ref{lem:aut:invar:minor}, it follows that~$\aut(G_\chi)$ is in $\Gamma_{f(h)}$.

\icase{3. $G_\chi$ is not edge-transitive and not vertex-transitive}

Let~$O_1$ be a vertex $\aut(G_\chi)$-orbit of minimum size, and
let~$O_2$ be a vertex $\aut(G_\chi)$-orbit adjacent to~$O_1$. Let~$\hat E$ be an edge $\aut(G_\chi)$-orbit whose edges have end points in both~$O_1$ and~$O_2$. Consider the graph~$\hat G\coloneqq(V(G),\hat E)$ induced by~$\hat E$. 
Consider the set $\CZ$ of connected components of~$\hat G$.
For each connected component $Z\in\CZ$ of~$\hat G$,
we define the vertex-colored graph $G_{Z,\chi}\coloneqq(G[Z],\chi')$
where $\chi'(v)\coloneqq (\chi(v),i)$ for $v\in Z$ and $i\in\{1,2\}$ is defined such that $v\in O_i$.
Then, for each connected component $Z\in\CZ$ of~$\hat G$ the vertex-colored graph $G_{Z,\chi}$ has exactly two $\aut(G_{Z,\chi})$-orbits $O_1\cap Z\neq O_2\cap Z$,
and for these two orbits it holds that $|O_1\cap Z|\leq |O_2\cap Z|$.
Now, consider the $\aut(G_\chi)$-invariant minor $G^*_\chi\coloneqq(G^*,\chi^*)$ that is obtained from $G$ by contracting the connected components of $\hat G$
where $\chi^*$ is the vertex-coloring induced by $\chi$ (as defined in the previous case).
Note that $O^*\coloneqq\CZ$ is the unique smallest $\aut(G^*_\chi)$-orbit of $G^*$.
By induction, for each branch set $Z\in\CZ$ it holds that $\aut(G_{Z,\chi})[O_1\cap Z]\in\Gamma_{f(h)}$
and also for $G^*_\chi$ it holds that $\aut(G^*_\chi)[O^*]\in\Gamma_{f(h)}$.
By Lemma~\ref{lem:aut:invar:minor}, it follows that~$\aut(G_\chi)[O_1]$ is in $\Gamma_{f(h)}$. 
\end{proof}

\begin{lem}\label{lem:aut:of:disconnected:graphs}
If~$G_\chi=(G,\chi)$ is a vertex-colored disconnected graph, then~$\aut(G_\chi)$ is a direct product of wreath products with the symmetric group as the top group of the automorphism groups induced on the connected components of~$G$.
\end{lem}

\begin{proof}
If the graph consists of~$t$ isomorphic connected graphs, we obtain precisely the wreath product with the symmetric group~$S_t$. Partitioning the components by isomorphism type, we obtain the direct product of the corresponding wreath products.
\end{proof}

We are now ready to state our main theorem, which employs the function~$f$ that exists by Lemma~\ref{lem:smallest:orbit:is:gamma:d}. Recall that~$\Theta_{d}$ is a restricted class of groups defined via certain repeated extensions using as building blocks symmetric groups and groups whose non-abelian composition factors are subgroups of~$S_d$ (as outlined in Section~\ref{sec:prelims}).

\begin{theo}\label{thm:kh:minor:free:means:Theta:group}
If~$G$ is a~$K_{h+1}$-minor-free graph, then~$\aut(G)\in \Theta_{f(h)}$.
\end{theo}
\begin{proof}
We show the more general statement for vertex-colored graphs~$G_\chi=(G,\chi)$ by induction on the size of~$G$
and the number of color classes.

If~$G$ is disconnected, then by induction and by Lemma~\ref{lem:aut:of:disconnected:graphs}, the automorphism group is a direct product of wreath products with symmetric (top) groups and (base) groups in~$\Theta_{f(h)}$.

If there is an $\aut(G_\chi)$-orbit $O$ of size exactly one, we can get rid of this orbit
by removing this vertex of $O$ from the graph and by coloring the neighbors accordingly.
More precisely, we define a vertex-colored graph $G_\chi'\coloneqq(G-O,\chi')$ where
$\chi'(v)\coloneqq(\chi(v), N_G(v)\cap O)$ for $v\in V(G)\setminus O$.
Now, the group $\aut(G_\chi)$ is isomorphic to $\aut(G_\chi')$, and thus the theorem follows by induction.

We can thus assume that the smallest $\aut(G_\chi)$-orbit has size at least 2 and that $G$ is connected.
Let~$O$ be an orbit of minimum size. The group~$\aut(G_\chi)$ acts naturally on~$O$ via the induced action. 
By Lemma~\ref{lem:smallest:orbit:is:gamma:d}, we know that the induced group on the orbit is in~$\Gamma_{f(h)}$. Consider the kernel of this homomorphism. It suffices to show that this kernel is in~$\Theta_{f(h)}$ since~$\aut(G_\chi)$ is an extension of this kernel by the~$\Gamma_{f(h)}$-action on~$O$.
The kernel consists of the automorphisms that fix all points in~$O$. We individualize all vertices in that orbit by refining the coloring.
More precisely, we define a vertex-colored graph $G_\chi'\coloneqq(G,\chi')$
where $\chi'(v)=(\chi(v),0)$ for all $v\in V(G)\setminus O$ and $\chi'(v)=(v,1)$ for all $v\in O$.
Then, the kernel is equal to $\aut(G_\chi')$, and is in $\Theta_{f(h)}$ by induction.
\end{proof}

\section{Babai's conjectures}\label{sec:babai:conj}

We now state three of Babai's conjectures that can for example be found in~\cite{MR633647}, and argue that our structural analysis shows that indeed each conjecture holds.

\begin{theo}
There is a function~$f$ such that a composition factor of the automorphism group of a graph of Hadwiger number at most~$h$ is cyclic, alternating or has order at most~$f(h)$. 
\end{theo}

\begin{proof}
This follows directly from Theorem~\ref{thm:kh:minor:free:means:Theta:group}
since groups in~$\Theta_{f(h)}$ have the desired composition factors.
\end{proof}

\begin{theo}
Only finitely many non-cyclic simple groups are represented by graphs of bounded Hadwiger number. 
\end{theo}
\begin{proof}
Again this follows directly from Theorem~\ref{thm:kh:minor:free:means:Theta:group} since
non-cyclic simple groups in~$\Theta_{f(h)}$ have bounded order.
\end{proof}

The third conjecture states that if the order of the automorphism group of a graph of bounded Hadwiger number does not have small prime divisors, then the group is a repeated direct and wreath product of abelian groups.

Recall the parameter $\alpha_h\in\CO(h\cdot\sqrt{\log h})$ from
\cref{theo:1} bounding the average degree in~$K_{h+1}$-minor-free graphs. Also recall that a permutation group is semi-regular if only the identity has a fixed point and regular if it is additionally transitive.

\begin{lem}\label{lem:smallest:orbit:is:abelian}
Let~$G$ be a connected, $K_{h+1}$-minor-free graph.
Let $\Delta\leq\aut(G)$ be a subgroup such that all
prime factors of $|\Delta|$
are greater than $\max\{\alpha_h,2\}$,
and let $O$ be a minimum cardinality $\Delta$-orbit.
Then, the induced group $\Delta[O]$ is regular and abelian.
Furthermore, if $\Delta$ is fixed-point free, then there is a fixed-point-free element $\delta\in\Delta$.
\end{lem}

\begin{proof}

We can assume that $|V(G)|> 1$, otherwise we are done.
\icase{1. $\Delta$ is transitive}
Note that in this case the graph $G$ is vertex-transitive since $\Delta\leq\aut(G)$.

We argue first that~$\Delta$ is regular, i.e., that all point
stabilizers are trivial (in addition to transitivity). Indeed, if this
were not the case, then since~$G$ is connected there would be a vertex~$v$ so that in the
point stabilizer~$\Delta_v\leq\Delta$ there is an automorphism that
moves neighbors of~$v$.
However, Theorem~\ref{theo:1} implies that the number of neighbors of~$v$ is at
most~$\alpha_h$ (since $G$ is regular),
which would lead to an
automorphism whose order has a prime factor of size at most~$\alpha_h$.

The regularity also implies that there is a fixed-point-free element $\delta\in\Delta$, in fact all non-trivial elements in~$\Delta$ are fixed-point free.

\begin{claim}\label{claim:cycle}
Each edge orbit under $\Delta\leq\aut(G)$ is a disjoint union of cycles
of length at least $\alpha_h$.
\end{claim}

\begin{claimproof}
Consider now an edge~$e=x_1x_2\in E(G)$ with a direction~$(x_1,x_2)$ say. Let~$E_{\rightarrow}$ be the $\Delta$-orbit of
the directed edge $(x_1,x_2)$ under~$\Delta$.
Since $\Delta$ is transitive, for every vertex~$v$ there is at least one directed edge $(v,w)$ for some $w\in V(G)$ in the $\Delta$-orbit of~$(x_1,x_2)$. Since~$\Delta$ is regular, there is at most one such edge.
Thus, the (directed) graph~$(V(G),E_{\rightarrow})$ is a disjoint union of (directed) cycles.
These cycles all have length at least~$\alpha_h$,
due to the absence of small factors in~$|\Delta|$. Since~$e$ was arbitrary, the statement holds for all edge $\Delta$-orbits.
\end{claimproof}

Let~$\hat{E}\subsetneq E(G)$ be a maximal union of edge orbits under $\Delta$ so
that~$\hat G \coloneqq (V(G),\hat{E})$
is not connected (possibly~$\hat{E}$ is empty).
If $\hat G$ is edgeless, then by \cref{claim:cycle}, the graph $G$ contains an edge $\Delta$-orbit that is
one single cycle, and thus $\Delta\leq\aut(G)$ is a cyclic group.
In the following, we thus assume that $\hat G$ has at least one edge.
Note that all components of~$\hat G$ induce isomorphic graphs.
Let $\CZ$ be the set of connected components of $\hat G$,
and let $G^*\coloneqq G/\hat E$ be the minor that is obtained by contracting all edges in~$\hat E$
(or equivalently all connected components $Z\in\CZ$).
Let~$e_0\in E(G)$ be an edge not in~$\hat{E}$ (so that $e_0$ has endpoints in two distinct connected components $Z,Z'\in\CZ$) and let~$E_0\subseteq E(G)$ be its $\Delta$-orbit.

\begin{claim}\label{claim:match}
  There is a cyclic order	 $Z_1,\ldots,Z_n\in\CZ$ of the connected
  components such that all edges in $E_0$ lie between two consecutive connected
  components $Z_i,Z_{i+1}$ (where $Z_{n+1}\coloneqq Z_1$). Moreover, the edge set $E_0$
  induces a perfect matching between each pair of consecutive connected
  components~$Z_i,Z_{i+1}$.
\end{claim}

\begin{claimproof}
  Consider the natural homomorphism $g\colon\aut(G)\to\aut(G^*)$ with image
  $\Delta^* \coloneqq g(\Delta) \leq\aut(G^*)$.  Note that
  $\Delta^*\leq\aut(G^*)$ is transitive. Since $\Delta^*$ is a
  homomorphic image of~$\Delta$, all prime factors of $|\Delta^*|$ are
  greater than $\alpha_h$. Moreover, the factor graph $G^*$
  is a minor of $G$, and thus it is $K_{h+1}$-minor free as well.
  Thus,
  \cref{claim:cycle} also holds for $G^*$ and $\Delta^*$.  For this
  reason, the image of $E_0$ in $G^*$ is a disjoint union of
  cycles.  However, by the definition of $\hat E$, it must be one
  single cycle (otherwise $E_0$ would have been added to~$\hat E$).  This shows the
  desired cyclic ordering of the connected components in $\CZ$.

  In order to show that $E_0$ induces a matching, we
  consider the bipartite graph that is induced by $E_0$ on two
  consecutive connected components $Z_i,Z_{i+1}$.  On the one hand,
  all vertices in this bipartite graph have degree at most 1 (by
  \cref{claim:cycle}). (We use here that~$n> 2$ since~$n\geq 
  \max\{\alpha_h,2\}$.) On the other hand, there are no isolated
  vertices since each vertex in $Z_i\cup Z_{i+1}$ is adjacent to some edge in
  $E_0$ (by transitivity of $\Delta$).
\end{claimproof}

By \cref{claim:match}, the factor graph $G^*$ contains an edge
$\Delta^*$-orbit that is a cycle.
We say that a connected component $Z$ \emph{has a spanning cycle}
if there is an edge $\Delta$-orbit that induces a cycle on $Z$.

\icase{1.a. There is a spanning cycle for some (and thus for all) $Z\in\CZ$}

Let $E_c$ be an edge $\Delta$-orbit inducing a cycle on $Z$, let $e_c=x_1x_2\in E_c$ be an edge
with an orientation~$(x_1,x_2)$ say, and let $E_{\rightarrow}$ be the $\Delta$-orbit of $(x_1,x_2)$.
Clearly, the (directed) edge set $E_{\rightarrow}$ induces a directed cycle on each $Z\in\CZ$.

Suppose that $E_{\rightarrow}\cup E_0$ is \emph{locally a grid},
i.e., there are vertices $v,w\in Z,v',w'\in Z'$ with
$(v,w),(v',w')\in E_{\rightarrow}$ and $vv',ww'\in E_0$. Take an automorphism
$\delta\in\Delta$ that maps $v$ to $w$ (and thus $v'$ to $w'$ since $\delta$ maps each $Z\in\CZ$ to itself), and
take an automorphism $\delta'\in\Delta$ that maps $v$ to $v'$ (and
thus $w$ to $w'$). Then, it holds that
$v^{\delta\delta'}=w'=v^{\delta'\delta}$.  By regularity of $\Delta$,
we have that $\delta\delta'=\delta'\delta$, and thus the automorphisms
commute.
Note that $\delta$ and $\delta'$ generate $\Delta$
since the generated group is transitive on $V(G)$ and $\Delta$ is regular.
Thus, the group
$\Delta$ is abelian.

Now, suppose $E_{\rightarrow}\cup E_0$ is not locally a grid. We construct a minor as follows.
Recall that $n=|V(G^*)|>\alpha_h$ and that $m\coloneqq |Z|>\alpha_h$ for $Z\in\CZ$ (by \cref{claim:cycle}).
Let $Z_1,\ldots,Z_n\in\CZ$ be the connected components in their cyclic order, i.e.,
$Z_i,Z_{i+1}$ are matched via $E_0$ (where $Z_{n+1}=Z_1$).
First, we delete all edges between $Z_n$ and $Z_1$.
This leads to $m$ vertex-disjoint $E_0$-paths $P_1,\ldots,P_m$ with $n$ vertices (and $n-1$ edges) each.
We define $H$ as the minor with $m$ vertices
that is obtained by contracting each path $P_1,\ldots,P_m$
to a single vertex.
\begin{claim}\label{claim:twistedMinor}
The minor $H$ has average degree greater than $\alpha_h$.
\end{claim}

\begin{claimproof}
  Let $C_1\coloneqq v_1,\ldots,v_m$ be the (directed) $E_{\rightarrow}$-cycle in
  $Z_1$.  By possibly renaming the indices of the paths, we can
  assume that $v_i$ belongs to $P_i$ for all $i\in[m]$.  For all
  $i\in[n]$ let $C_i$ be the (directed) $E_{\rightarrow}$-cycle in $Z_i$.  The existence of the
  matching between $C_1$ and $C_2$ implies that if an automorphism
  $\delta\in\Delta$ rotates the cycle $C_1$ by one, it rotates the
  cycle $C_2$ by some positive integer $k\geq 1$ (and in fact
  $k\geq 2$ since we do not have a local grid).  By the regularity of
  $\Delta$, the integer $k$ is co-prime to $m$.  The fact that
  $\Delta$ acts cyclic on $\CZ$ implies that consecutive cycles
  $Z_i,Z_{i+1}$ are matched isomorphically,
  and thus if $C_i$ is rotated by one, then $C_{i+1}$ is rotated by
  $k$ (for the same $k$ as above).  In general, if an automorphism rotates $C_1$ by
  one, then for all $i\in[n]$ the cycle $C_i$ is rotated by $k^{i-1}$
  modulo $m$.  Now, let $n_0$ be the largest integer such that
  $k^0,k^1,k^2,\ldots,k^{n_0-1}$ are pairwise distinct modulo $m$.
  Since $k$ is co-prime to $m$ and since $k^{n}\equiv k^0\mod m$,
  it follows that $n_0$ is a divisor of $n$, and
  thus it holds that $n_0>\alpha_h$.  Finally, $H$ can be described as
  a graph with vertex set $v_1',\ldots,v_m'$ and an edge set $E(H)$ that
  is a union of $n_0$ cycles $C_1',\ldots,C_{n_0}'$ being pairwise
  edge-disjoint (when viewed as directed cycles).  Then, each vertex in $H$ has degree $2n_0>\alpha_h$
  (in particular, $(v_1',v_j')$ is a directed edge in $C_i'$ if and
  only if $j-1\equiv k^{i-1}\mod m$).
\end{claimproof}

Combining \cref{claim:twistedMinor} with \cref{theo:1} implies
that $G$ has a $K_{h+1}$ minor.

\icase{1.b. There is no spanning cycle for $Z\in\CZ$}

In the remaining case, the graph $\hat G$ has connected components
$Z\in\CZ$ that do not contain spanning cycles.
We show that also in
this case $G$ has a $K_{h+1}$ minor. By \cref{claim:match}, there are
(consecutive) connected components $Z,Z'\in\CZ$ and an edge orbit
$E_0$ that induces a matching between $Z$ and~$Z'$. Our strategy is to find a collection of disjoint cycles in~$Z$ and a collection of disjoint cycles in~$Z'$ so that each cycle in~$Z$ is adjacent to many cycles on~$Z'$ via the matching and vice versa. Contracting the cycles will give a minor of large average degree.

For a set $E\subseteq E(G)$ and $Z\in\CZ$,
let $\CZ_{E,Z}$ be the set of
connected components of $(V(G),E)[Z]$.
Pick an edge
$e'\subseteq Z'$ such that the number of connected components of the edge $\Delta$-orbit $E'$ of~$e'$ restricted to $Z'$ is as
small as possible, i.e., $|\CZ_{E',Z'}|$ is minimal.  Assume that $e'=v'w'$ and let $v,w\in Z$ be the
vertices that are matched via $E_0$ with $v'$ and~$w'$, respectively.  Let
$F$ be the $\Delta$-orbit of $vw$ and consider its restriction to~$Z$ (and note that the set $F$ might consist of non-edges).
Note that $|\CZ_{F,Z}|=|\CZ_{E',Z'}|>1$ since $Z'$ has no spanning cycle.
Therefore, we can find an edge
$\Delta$-orbit $E\subseteq E(G)$ such that each connected component
in $\CZ_{E,Z}$ intersects at least two connected components in $\CZ_{F,Z}$.
On the other hand, also each connected component of $\CZ_{F,Z}$ intersects
at least two connected components in $\CZ_{E,Z}$
since $|\CZ_{F,Z}|=|\CZ_{E',Z'}|\leq|\CZ_{E,Z}|$ by the minimal choice of $E'$.
We define the minor $H$ of $G$ by
restricting the graph to $Z\cup Z'$ and contract all edges in $E$ and
all edges in $E'$ (the edges of $Z'$ matched to $F$).
\begin{claim}
The minor $H$ has average degree greater than $\alpha_h$.
\end{claim}
\begin{claimproof}
It suffices to show that each connected component in $\CZ_{E,Z}$ intersects more than $\alpha_h$
connected components in $\CZ_{F,Z}$,
and each connected component in $\CZ_{F,Z}$ intersects more than $\alpha_h$ connected components in $\CZ_{E,Z}$.	
If $Z\in\CZ_{F,Z}$ intersects the connected components $Z_1,\ldots,Z_t\in \CZ_{E,Z},t\geq 2$,
then there is a permutation $\delta\in\Delta$ that rotates the $F$-cycle in $Z$ (stabilizing $Z$ setwise) and permutes $Z_1,\ldots,Z_t$
non-trivially. But since $|\Delta|$ has only prime factors greater than $\alpha_h$,
it follows that $t>\alpha_h$.
Symmetrically, the same argument can also be applied when the roles of $\CZ_{F,Z}$ and $\CZ_{E,Z}$ are swapped.
\end{claimproof}

Thus, the minor $H$ has average degree greater than $\alpha_h$, and contains $K_{h+1}$ as a minor.

\icase{2. $\Delta$ is not transitive}

Let $O_1$ be a minimum cardinality $\Delta$-orbit
and pick a second $\Delta$-orbit $O_2$ distinct from $O_1$
such that $O_1\cup O_2$ contains an edge $e$ with endpoints in $O_1$ and $O_2$.
Let $\hat E\subseteq E(G)$ be the edge $\Delta$-orbit of $e$,
let $\hat G\coloneqq(O_1\cup O_2,\hat E)$ be the subgraph of $G$ induced on $O_1\cup O_2$ and $\hat E$,
and let $\CZ$ be the set of connected components of
$\hat E$.

\begin{claim}
For all connected components $Z\in\CZ$ the induced graph $\hat G[Z]$ is a star (i.e., the complete bipartite graph~$K_{1,t}$ for some~$t$) with center in $O_1$.
\end{claim}

\begin{claimproof}
Assume for the sake of contradiction that there is a vertex $v_2\in Z\cap O_2$
such that $\deg_{\hat G[Z]}(v_2)>1$.
Then, since all edges of $\hat E$ are in the same $\Delta$-orbit,
there is an automorphism $\delta\in\Delta$
that fixes $v_2$ and acts non-trivially on $N_{\hat G[Z]}(v_2)\subseteq O_1$.
Since the order of $\delta\in\Delta$ does only have prime factors
greater than $\alpha_h$, this implies that
$\deg_{\hat G[Z]}(v_2)>\alpha_h$.
Since $\hat G[Z]$ is biregular and since $|O_1|\leq|O_2|$,
we have that $\deg_{\hat G[Z]}(v_1)\geq \deg_{\hat G[Z]}(v_2)>\alpha_h$ for all $v_1\in Z\cap O_1$.
This contradicts the fact that the average degree of $\hat G[Z]$
is at most $\alpha_h$ (\cref{theo:1}).
\end{claimproof}

Let $G^*\coloneqq G/\CZ$ be the minor of
$G$ that is obtained by contracting all connected components in $\CZ$.
Consider the homomorphism $g\colon\aut(G)\to\aut(G^*)$
with image $\Delta^*\coloneqq g(\Delta)$.
Note that $O^*\coloneqq\CZ$ is a minimum cardinality $\Delta^*$-orbit.
As homomorphic image of~$\Delta$, the order of the subgroup $\Delta^*\leq\aut(G^*)$ has
only prime factors greater than~$\alpha_h$.
By induction,
we conclude that $\Delta^*[O^*]$ is regular and abelian.
Note that $\Delta^*$ is isomorphic to $\Delta[V(G)\setminus O_2]$ (as a permutation group),
and thus $\Delta[O_1]$ is regular and abelian as well.

Furthermore, if $\delta^*\in\Delta^*\leq\aut(G^*)$ is fixed-point free,
then each element in the preimage $g^{-1}(\delta^*)\subseteq\Delta\leq\aut(G)$
is fixed-point free.
\end{proof}

Regarding the parameter $\alpha_h$ from \cref{theo:1}, note that the complete bipartite graph $K_{h,h}$ is $K_{h+1}$-minor free and has average degree $h$,
and thus $\alpha_h\geq h$.

\begin{theo}
Let~$G$ be a~$K_{h+1}$-minor-free graph
such that all
prime factors of $|\aut(G)|$ are greater than
$\max\{\alpha_h,2\}$.
Then, the automorphism group $\aut(G)$ is a repeated direct and wreath product of abelian groups. 
\end{theo}
\begin{proof}
To facilitate induction over~$|V(G)|$, we prove the statement for vertex-colored
graphs $G_\chi=(G,\chi)$.
In the base case when $|V(G)|=1$, there is nothing to show.

If~$G$ had isomorphic connected components, then~$\Aut(G_\chi)$ would have
an automorphism of order 2. Thus, the automorphism group $\Aut(G_\chi)$ is the direct product of
the automorphism groups of the connected components.
We can thus assume that~$G$ is connected.

By applying \cref{lem:smallest:orbit:is:abelian} to the (uncolored) graph $G$ with $\Delta\coloneqq\aut(G_\chi)$,
we conclude that there is a (minimal) $\Aut(G_\chi)$-orbit $O$ such that
$\aut(G_\chi)[O]$
is abelian. Let~$\aut(G_\chi)_{(O)}$ be the pointwise stabilizer of~$O$.
Consider the fixed points $F\supseteq O$ of $\aut(G_\chi)_{(O)}$.
By definition of~$F$, the induced group $\aut(G_\chi)[F]$ is isomorphic to
$\aut(G_\chi)[O]$,
and thus abelian.

Let us observe that~$\aut(G_\chi)[F]$ is semi-regular as follows. The size of an orbit in~$F$ cannot be larger than~$|O|$ since~$\aut(G_\chi)[O]$ is regular. It cannot be smaller than~$|O|$ since~$|O|$ is minimal.
This implies that~$\aut(G_\chi)[F]$ acts regularly on each orbit because transitive abelian permutation groups are always regular.

Let $\CZ$ be the connected components of
$G-F$.

\begin{claim}\label{claim:small}
$|N_{G}(Z)|\leq h$ for each $Z\in\CZ$.
\end{claim}

\begin{claimproof}
Assume for the sake of contradiction that $|N_{G}(Z)|> h$.
We construct a $K_{h+1}$ minor as follows.
Since the automorphism group order $|\aut(G_\chi)|$ does not have 2 as a prime factor,
the group $\aut(G_\chi)_{(O)}\leq\aut(G_\chi)$ fixes all connected components
setwise (otherwise there is a permutation in $\aut(G_\chi)_{(O)}$ that swaps two components in $\CZ$ which must have even order).
The group~$\aut(G_\chi)_{(O)}$ induces on~$Z\in\CZ$ a group $\Delta_Z\coloneqq\aut(G_\chi)_{(O)}[Z]$.
Note that $\Delta_Z\leq\aut(G_\chi[Z])$ and, the group being a restriction of~$\aut(G_\chi)_{(O)}$, the order of $\Delta_Z$ does only have prime factors
greater than $\alpha_h$.
Furthermore, the subgroup $\Delta_Z$ does not have fixed points since all fixed points of $\aut(G_\chi)_{(O)}$
are in $F$, which is disjoint from $Z$.
By \cref{lem:smallest:orbit:is:abelian},
there is a fixed-point-free permutation $\delta_Z\in\Delta_Z$.
In the following, we construct a tree in $G[Z\cup N(Z)]$ that contains at most
one vertex from each $\Delta_Z$-orbit and whose set of leaves is 
exactly the set $N(Z)$.
This can be done greedily: initially, we pick an arbitrary vertex $v\in N(Z)$
and add this vertex to the tree $T$, i.e., define $V(T)\coloneqq\{v\}$.
While there is a vertex $w\in N(Z)$
that is not yet contained in $V(T)$, we extend the tree $T$
by adding a shortest path in $G[Z\cup\{w\}]$ from $V(T)$ to $w$.
Note that if $v,v'\in O_Z$ are in the same $\Delta_Z$-orbit $O_Z$,
then $v$ and $v'$ have the same distance to $w$.
For this reason, the set $V(T)$ does not contain two vertices in the same $\Delta_Z$-orbit
throughout the construction.

Having defined~$T$, we apply the fixed-point-free permutation $\delta_Z$
to $T$ and define $T_i\coloneqq T^{(\delta_Z)^i}$ for $i\in[h+1]$.
Since $\delta_Z$ has no fixed points on $Z$ and since $|\Delta_Z|$ has only prime factors greater than $\alpha_h\geq h$,
all $\delta_Z$-orbits have size greater than $\alpha_h\geq h$. Thus, the trees
$T_1,\ldots,T_{h+1}$ have inner vertices that are pairwise disjoint and they have the same set of leaves, namely $N(Z)$. This gives a $K_{h+1,h+1}$-minor
of $G[Z\cup N(Z)]$, and thus
a $K_{h+1}$ minor.
\end{claimproof}

Let $G_{Z,\chi}$ be the vertex-colored graph based on $G[Z]$
where each vertex in $Z$ is colored with its $\Delta$-orbit
(i.e., two vertices $v,w\in Z$ get the same color
if and only if $v,w$ are in the same $\Delta$-orbit).

\begin{claim}\label{claim:extends}
Each automorphism
of $G_{Z,\chi}$ can be extended to an automorphism of $G_\chi[Z\cup N(Z)]$ that fixes $N(Z)$ pointwise.
\end{claim}

\begin{claimproof}
Suppose $\delta_Z\in\aut(G_{Z,\chi})$. We extend $\delta_Z$ to $\delta_{\hat Z}\in\aut(G_\chi[Z\cup N(Z)])$
by fixing all points in $N(Z)$.
We claim that $\delta_{\hat Z}$ preserves all edges between $Z$ and $N(Z)$.
Let $vw\in E(G)$ be an edge
with $v\in Z,w\in N(Z)$.
Since $v,v^{\delta_{\hat Z}}$ are in the same $\Delta$-orbit (because of the coloring of $G_{Z,\chi}$),
there is an automorphism $\delta\in\Delta$ such that $v^\delta=v^{\delta_{\hat Z}}$.
Since $\delta$ maps $Z$ to itself, it also stabilizes $N(Z)$ setwise,
and since $|N(Z)|\leq h$ (\cref{claim:small}), it fixes $N(Z)$ pointwise (since the order of $\delta$
has only prime factors greater than $\alpha_h\geq h$).
This means that $w^\delta=w$ for all $w\in N(Z)$.
But since $\delta$ is an automorphism and $vw\in E(G)$, it holds that $v^{\delta_{\hat Z}} w^{\delta_{\hat Z}}=v^{\delta_{\hat Z}} w=v^\delta w^\delta\in E(G)$.
This proves the claim.
\end{claimproof}

By \cref{claim:extends}, we have that $\aut(G_{Z,\chi})=\Delta_Z$,
and thus $\aut(G_{Z,\chi})$ is a homomorphic image of $\aut(G_\chi)$
and only has prime factors greater than $\alpha_h$.
This allows us to apply induction,
and thus $\aut(G_{Z,\chi})$ is a repeated direct and wreath product of abelian groups.
We need to show that $\aut(G_\chi)$ is a repeated direct and wreath product of abelian groups.

We can assume that in the vertex-colored graph~$G_\chi=(G,\chi)$ vertices from different $\Delta$-orbits have different colors.

Recall that~$\aut(G_\chi)[F]$ is semi-regular.
It follows from Claim~\ref{claim:small} that if an automorphism of $G_\chi$ maps~$Z$ to itself, then it fixes~$N(Z)$ pointwise, and by semi-regularity the entire set~$F$ pointwise.
Therefore, the automorphism group $\aut(G_\chi)$ acts semi-regularly on $\CZ$.
This implies that every graph $G_{Z,\chi},Z\in\CZ$ is isomorphic to exactly $|\aut(G_\chi)[F]|$ graphs $G_{Z',\chi},Z'\in\CZ$.

Let~$\tilde Z$ be a maximal union of
connected components from~$\CZ$ such that the graphs $G_{Z,\chi}$ for $Z\in \tilde Z$ are pairwise non-isomorphic. (This simply means that the
vertices in the different components have different colors.)
The images of~$\tilde Z$ under automorphisms from~$\aut(G_\chi)$ are pairwise
disjoint. They are permuted by the automorphisms~$\aut(G_\chi)$. There
are exactly $|\aut(G_\chi)[F]|$
different images and the induced permutation
group on this set of images is $\aut(G_\chi)[F]$.
We will now use Claim~\ref{claim:extends} to show that~$\aut(G_\chi)$ is the wreath
product~$\aut(G_{\tilde Z,\chi})\wr\aut(G_\chi)[F]$.

Let $\tilde Z_1,\ldots,\tilde Z_{|O|}$
be the images of $Z$ under $\aut(G_\chi)$.
Recall that $\Psi\coloneqq\aut(G_\chi)_{(F)}$ stabilizes each $Z\in\CZ$ setwise,
and thus $\Psi\trianglelefteq\aut(G_\chi)$ is a direct
product $\aut(G_{\tilde{Z}_1,\chi})\times\ldots\times\aut(G_{\tilde{Z}_{|O|},\chi})$ of isomorphic (base) groups (by Claim~\ref{claim:extends} all combinations of automorphisms for the graphs~$G_{\tilde{Z}_1,\chi},\ldots,G_{\tilde{Z}_{|O|},\chi}$ extend to automorphisms of~$G$).

In the following, we define a suitable (top) group $\Theta\leq\aut(G_\chi)$
permuting the components $\tilde Z_1,\ldots,\tilde Z_{|O|}$.
For each~$i\in \{1,\ldots,|O|\}$ let $\phi_{1,i}$ be an isomorphism from $G_{\tilde Z_1,\chi}$ to $G_{\tilde Z_i,\chi}$. Also define $\phi_{i,j}\coloneqq\phi_{1,i}^{-1}\phi_{1,j}$ for all $i,j\in\{1,\ldots,|O|\}$.
For~$\varphi\in \aut(G_\chi)[F]$ choose $\hat\phi_0\in\aut(G_\chi)$ such that $\hat\phi_0[F]=\phi$.
Then, for each $i\in\{1,\ldots,|O|\}$ there is an $i'\in\{1,\ldots,|O|\}$
such that $\tilde Z_i^{\hat\phi_0}=\tilde Z_{i'}$.
Note that $i'$ only depends on $\phi$ and $i$ (but not on the choice of $\hat\phi_0$)
since otherwise there is a permutation in $\aut(G_\chi)_{(F)}$ that swaps
two components in $\CZ$ contradicting that
$|\aut(G_\chi)_{(F)}|$ is odd.
We define $\hat\phi$ such that $\hat\phi[F]=\phi$
and $\hat\phi[\tilde Z_i]=\phi_{i,i'}$.
Then, it holds that $\hat\phi\in\aut(G_\chi)$ since
it holds that $(\hat\phi\hat\phi_0^{-1})[F]=\id_F$ and $(\hat\phi\hat\phi_0^{-1})[Z]\in\aut(G_{Z,\chi})$ for
each $Z\in\CZ$ implying that
$\hat\phi\hat\phi_0^{-1}\in\aut(G_\chi)$ by \cref{claim:extends}.
We define the (top) group $\Phi\leq\aut(G_\chi)$ (isomorphic to $\aut(G_\chi)[F]$) as the set
of extensions $\{\hat\phi\mid \phi\in\aut(G_\chi)[F]\}$.
Then, the groups $\Psi,\Phi$ are permutable complements, i.e.,
$\Psi\cap\Phi$ is the trivial group and $\Psi\Phi=\aut(G_\chi)$.
Furthermore, the top group $\Phi$ acts as automorphism on $\Psi$ by conjugation.
Thus, the automorphism group $\aut(G_\chi)$ is an (internal) wreath product of $\Psi$ and~$\Phi$.
\end{proof}

\section{Conclusion}\label{sec:conclusion}
We characterized the automorphism groups of
graphs of bounded Hadwiger number. The characterization lends itself
to proving various properties such as the resolution of Babai's three conjectures. A central part of the characterization
analyzes edge-transitive graphs, and this is done via limit
constructions.
 
However, this approach does not lead to explicit bounds and it remains
as interesting future work to analyze how large the graphs have to be
for the structural requirements to kick in. For example, it might be
interesting to determine reasonable bounds for the function~$f(h)$ in
the classification theorems.  In particular, it remains open what
quantitative results can further be concluded for vertex- or
edge-transitive graphs.

In our proofs we focused on the possible groups that can arise as
automorphism groups of bounded Hadwiger number graphs. However, our
proofs actually show that the structure of the graphs is also very
restricted when the automorphism groups are sufficiently rich. It is
an interesting question whether we can make further use of the
structure that must necessarily emerge in the graphs.

\bibliographystyle{amsplain}


\providecommand{\bysame}{\leavevmode\hbox to3em{\hrulefill}\thinspace}
\providecommand{\MR}{\relax\ifhmode\unskip\space\fi MR }
\providecommand{\MRhref}[2]{%
  \href{http://www.ams.org/mathscinet-getitem?mr=#1}{#2}
}
\providecommand{\href}[2]{#2}

\begin{aicauthors}
\begin{authorinfo}[mgro]
  Martin Grohe\\
  RWTH Aachen University\\
  Aachen, Germany\\
  grohe\imageat{}informatik\imagedot{}rwth-aachen\imagedot{}de\\
  \url{https://www.lics.rwth-aachen.de/~grohe/}
\end{authorinfo}
\begin{authorinfo}[pschw]
  Pascal Schweitzer\\
  Professor\\
  TU Darmstadt\\
  Darmstadt, Germany\\
  schweitzer\imageat{}mathematik\imagedot{}tu-darmstadt\imagedot{}de \\
  \url{https://www.mathematik.tu-darmstadt.de/~schweitzer/}
\end{authorinfo}
\begin{authorinfo}[dwieb]
  Daniel Wiebking\\
  Hannover, Germany\\
  daniel\imagedot{}wiebking\imageat{}web.de
\end{authorinfo}
\end{aicauthors}

\end{document}